\documentclass[12pt,reqno]{amsart}
\usepackage{mathtools}
\usepackage{empheq}
\usepackage{colonequals}

\usepackage{amsmath,amssymb,amsthm,microtype,geometry,mathscinet,enumitem,url,amsrefs}
\usepackage{parskip}
\usepackage{stackengine}
\usepackage{colonequals}
\usepackage{xcolor}
\usepackage[colorlinks=true,linkcolor=black,citecolor=black,urlcolor=black]{hyperref}

\usepackage[utf8]{inputenc}
\newcommand{\FS}{\operatorname{FS}}
\newcommand{\e}{\mathrm e}

\makeatletter
\NewCommandCopy\@@pmod\pmod
\DeclareRobustCommand{\pmod}{\@ifstar\@pmods\@@pmod}
\def\@pmods#1{\mkern4mu({\operator@font mod}\mkern 6mu#1)}
\makeatother

\numberwithin{equation}{section}

\DeclareMathAlphabet{\curly}{U}{rsfs}{m}{n}
\newtheorem{thm}{Theorem}[section]
\newtheorem{cor}[thm]{Corollary}
\newtheorem{lem}[thm]{Lemma}
\newtheorem{prop}[thm]{Proposition}
\theoremstyle{remark}
\newtheorem{rmk}{Remark}[section]

\usepackage{enumitem}

\def\Cc{\mathcal{C}}

\def\Bb{\mathcal{B}}

\def\Ff{\mathcal{F}}

\def\N{\mathbb{N}}
\def\Q{\mathbb{Q}}

\def\Pp{\mathcal{P}}
\def\Qq{\mathcal{Q}}

\def\E{\mathbb{E}}

\def\Z{\mathbb{Z}}

\def\PP{\mathbb{P}}
\def\R{\mathbb{R}}

\def\T{\mathbb{T}}

\renewcommand\subset\subseteq

\title[Strongly complete sets and a conjecture of Erd\H{o}s]{Strongly Complete sets and a conjecture of Erd\H{o}s}
\author{Steve Fan}
\address{Department of Mathematics\\ University of Georgia\\ Athens, GA 30602}
\email{Steve.Fan@uga.edu}


\begin{document}

\begin{abstract}
A set $A\subseteq\N$ is called \textit{complete} if every sufficiently large
integer can be written as a sum of distinct elements of $A$, and \textit{strongly complete} if it remains complete after finitely many elements are deleted. Building on recent work of Bergelson and Simmons and that of Griesmer, we establish a new strong-completeness criterion exploiting a three-component partition of a given set. As an application, we show that $A$ is strongly complete whenever $|A\cap(2^k,2^{k+1}]|\ge5$ for all sufficiently large $k\in\N$ or $|A\cap(x,2x]|\ge 3$ for all sufficiently large $x$, and 
\[
  \sum_{a\in A}\|a\theta\|=\infty,
  \quad\forall\theta\in\R\setminus\Z.
\]
In particular, this resolves a 1961 conjecture of Erd\H{o}s. The new strong-completeness criterion also enables us to make progress on a 1996 problem of Burr, Erd\H{o}s, Graham, and Li concerning strong completeness of mixed power sets by refining a previous result of Bergelson and Simmons. Besides, we study the polynomially perturbed ray set
\[
 \{\lfloor t\alpha^n\rfloor,\lfloor t\alpha^n\rfloor+P(n):n\in\N\},
\]
which combines the polynomial set $\{P(n):n\in\N\}$ and the single-ray set $\{\lfloor t\alpha^n\rfloor:n\in\N\}$ both previously considered by Graham, and show that it is strongly complete for any $t>0$ and $\alpha\in(0,2)$ and any primitive integer-valued polynomial $P$. The machinery developed for the proof of this result also yields other interesting applications.
\end{abstract}

\maketitle
%\tableofcontents

\section{Introduction}\label{S:Intro}

A classic problem in additive number theory concerns representing every sufficiently large integer as a sum of distinct elements of a given subset of integers. Following Burr and Erd\H{o}s \cite{BurrErdos81}, we call a set $A\subseteq\N$ \emph{complete} if every sufficiently large positive integer can be expressed as a sum of distinct elements of $A$. Thus, a complete set is a particularly rigid form of additive basis of which each element may be used as a summand at most once. As in \cite[Definition 1.1]{BergelsonSimmons17} we define 
\begin{equation}\label{eq:FS(A)}
\FS(A):=\left\{\sum_{a\in F}a:\emptyset\ne F\subseteq A\text{ finite}\right\},    
\end{equation}
which records all the nonempty subset sums of $A$. Then $A$ is complete if $\N\setminus\FS(A)$ is finite. Familiar examples of a complete set include 
\begin{enumerate}
    \item binary basis $\{2^k:k\ge 0\}$; \label{item:b-adic}
    \item $d\N\cup\{1,\ldots,d-1\}$, where $d\in\N_{\ge2}$;\label{item:dN}
    \item integer-valued polynomials $\{|P(n)|:P(n)\ne0\}$, where $P:\Z\to\Z$ is a nonconstant polynomial such that $|\{n\in\Z/p\Z:P(n)\equiv 0\pmod* p\}|<p$ for every prime $p$ \cite{Graham64a};
    \item odd primes \cite{BergelsonSimmons17};
    \item mixed bases $\{a^kb^{\ell}:k,\ell\in\Z_{\ge0}\}$, where $a,b\in\N_{\ge2}$ are coprime \cite{Birch59}.
\end{enumerate}
More recent examples and results can be found in \cite{BergelsonSimmons17}. We also say that $A$ is \emph{strongly complete} if $A\setminus B$ is complete for every finite subset $B\subseteq A$. Thus, a strongly complete set is automatically complete. On the other hand, a simple modular arithmetic argument shows that the sets in \eqref{item:b-adic} and \eqref{item:dN} are not strongly complete. 

The notion of strong completeness is also stronger in the sense that it automatically grants abundance of representations. Given $A\subseteq\N$, denote by $q_A(n)$ the number of representations of $n$ as sums of distinct elements of $A$. Then $A$ is strongly complete if and only if for every finite $F\subseteq A$, 
\begin{equation}\label{eq:q_A}
 \lim_{n\to\infty}q_{A\setminus F}(n)=\infty.
\end{equation}
The sufficiency part is trivial. To check the necessity part, fix $k\in\N$ and an arbitrary $k$-element subset $B\subseteq A\setminus F$. Completeness of $C=A\setminus(F\cup B)$ implies that if $n$ is sufficiently large, then
\[
 n-\sum_{b\in E}b\in\FS(C)
\]
for each $E\subseteq B$. There are $2^k$ choices for $E$, so $q_{A\setminus F}(n)\ge 2^k$. Since $k$ is arbitrary,
\eqref{eq:q_A} follows. As every representation of an integer $n\le x$ uses only elements of $A\cap[1,x]$, this argument also gives
\[2^{|A\cap[1,x]|}\ge\sum_{n\le x}q_{A}(n)\ge 2^kx +O_k(1),\]
whence
\[\lim_{x\to\infty}\left(|A\cap[1,x]|-\log_2 x\right)=\infty,\]
where $\log_c x:=\log x/\log c$ for any $c,x>0$ with $c\ne1$. The asymptotic analysis of $q_A(n)$ for general sets $A$, pioneered by Roth and Szekeres \cite{RothSzekeres54}, is a classical topic in the theory of integer partitions. 

In 1960 Cassels \cite{Cassels60} established a general completeness criterion combining local abundance in dyadic intervals with divergence of a Diophantine series. He proved that a set $A\subseteq\N$ is complete if it satisfies both the slow dyadic growth condition
\[\lim_{x\to\infty}\frac{|A\cap(x,2x]|}{\log\log x}=\infty\]
and the $L^2$-divergence 
\begin{equation}\label{eq:L2-divergence}
\sum_{a\in A}\|a\theta\|^2=\infty,
  \quad\forall\theta\in\T\setminus\{0\}.
\end{equation}
Here for every $x\in\R$, $\|x\|$ denotes the distance from $x$ to the nearest integer, or equivalently, the distance from $x$ to $0$ when $x$ is viewed as a point on the torus $\T=\R/\Z$. Cassels obtained his result by applying the Hardy--Littlewood circle method to the generating function of $q_A(n)$. He also remarked that it would be possible to weaken his hypotheses by refining the integral estimates in his proof, but he did not provide details. More recently, Bergelson and Simmons \cite{BergelsonSimmons17} proved an interesting completeness result which is close to being a generalization of Cassels's theorem.

Erd\H{o}s \cite[p. 231]{Erdos61} conjectured in 1961 that Cassels's hypotheses can be relaxed to the arbitrary dyadic growth condition
\begin{equation}\label{eq:erdos-growth}
  \lim_{x\to\infty}|A\cap(x,2x]|=\infty,
\end{equation}
combined with the $L^1$-divergence 
\begin{equation}\label{eq:L1-divergence}
  \sum_{a\in A}\|a\theta\|=\infty,
  \quad\forall\theta\in\T\setminus\{0\}.
\end{equation}
This is listed as Problem \#254 in Bloom’s online Erd\H{o}s problem list \cite{Bloom254}. It is not hard to see that every strongly complete set $A\subseteq\N$ satisfies \eqref{eq:L1-divergence}. Indeed, if there is some $\theta\in\T\setminus\{0\}$ such that the corresponding series converges, then we can find $N_0\in\N$ such that
\[\sum_{a\in A\cap(N_0,\infty)}\|a\theta\|<\frac{\|\theta\|}{2}.\]
Since $A$ is strongly complete, every sufficiently large integer belongs to $\FS(A\cap(N_0,\infty))$. By the triangle inequality,
\[
 \|\theta\|\le\|n\theta\|+\|(n+1)\theta\|<\frac{\|\theta\|}{2}+\frac{\|\theta\|}{2}=\|\theta\|
\]
whenever $n\in\N$ is sufficiently large, which is impossible.

Our first main result provides a strong-completeness criterion which only requires $A$ to have a few elements locally.
\begin{thm}\label{thm:main}
Let $\rho>1$ and define
\begin{equation}\label{eq:constants}
  u_\rho=\left\lceil\rho(\rho-1)\right\rceil,
  \qquad
  v_\rho:=\left\lceil\frac{\rho^3}{\rho+1}\right\rceil,
 \qquad 
 M_\rho:=\min\left\{2u_\rho+1,2v_\rho\right\}.
\end{equation}
Let $M\ge M_{\rho}$ be an integer, and suppose that $A\subseteq\N$ satisfies \eqref{eq:L1-divergence} and the inequality 
\begin{equation}\label{eq:growth}
\big|A\cap(\rho^k,\rho^{k+1}]\big|\ge M    
\end{equation}
for every sufficiently large $k\in\N$. Then
\begin{equation}\label{eq:q_A-growth}
\lim_{n\to\infty} \frac{q_{A\setminus F}(n)}
{n^{(M-M_\rho)\log_\rho2}}=\infty
\end{equation}
for every finite subset $F\subseteq A$. In particular, $A$ is strongly complete.
\end{thm}

It is worth noting that $M_{\rho}$ attains its minimum $2$ for every $1<\rho\le \rho_0$, where
\[
\rho_0:=\sqrt[3]{\frac{1}{2}+\frac{\sqrt{69}}{18}}+\sqrt[3]{\frac{1}{2}-\frac{\sqrt{69}}{18}}=1.324717957\ldots
\]
is the unique real solution to $\rho^3=\rho+1$.

Theorem \ref{thm:main} will be proved in Section \ref{S:Pf-main} as an application of a new three-component strong-completeness criterion (Theorem \ref{thm:strong-3set}) developed in Section \ref{S:strong-3set}. Our argument is flexible enough to yield an extension of this theorem by isolating the role played by the $\rho$-adic intervals in the proof. See Theorem \ref{thm:ordered-blocks}.

A variation of the argument, which we also present in Section \ref{S:Pf-main}, produces a variant of Theorem \ref{thm:main} under the hypothesis that elements of $A$ have an asymptotic fixed-step bounded ratio.

\begin{thm}\label{thm:bounded-ratio}
Let $\rho>1$ and put $r=2\lceil\rho\rceil-1$. Let $A=\{a_1<a_2<\cdots\}\subseteq\N$ be an infinite set satisfying \eqref{eq:L1-divergence} and 
\[a_{n+r}\le\rho a_n+o\left(\sum_{j=1}^na_j\right)\] 
for all sufficiently large $n$. Then $A$ is strongly complete. In particular, a set $A\subseteq\N$ is strongly complete if \eqref{eq:L1-divergence} holds and $|A\cap(x,\rho x]|\ge r$ for all sufficiently large $x$ (which is equivalent to $a_{n+r}\le\rho a_n$ for all sufficiently large $n$).
\end{thm}
%
%A growth rate for $q_{A\setminus F}(n)$ like \eqref{eq:q_A-growth} with $2\lceil\rho\rceil-1$ in place of $M_{\rho}$ also holds under the hypothesis that $|A\cap(x,\rho x]|\ge M\ge 2\lceil\rho\rceil-1$ for all sufficiently large $x$. 

Taking $\rho=2$ in Theorems \ref{thm:main} and \ref{thm:bounded-ratio} confirms Erd\H{o}s's conjecture in a considerably strong way. 

\begin{cor}\label{cor:erdos61}
Suppose that $A\subseteq\N$ satisfies \eqref{eq:L1-divergence} and one of the following conditions:
\begin{enumerate}
\item $|A\cap(2^k,2^{k+1}]|\ge5$ for all sufficiently large $k\in\N$.
\item $|A\cap(x,2x]|\ge 3$ for all sufficiently large $x$.
\end{enumerate}
Then $A$ is strongly complete.
\end{cor}

Denote by $M_{\rho}^{\ast}$ the least positive integer such that every set $A\subseteq\N$ satisfying \eqref{eq:L1-divergence} and 
\begin{equation}\label{eq:M_rho*}
 \big|A\cap(\rho^k,\rho^{k+1}]\big|\ge M_{\rho}^{\ast}   
\end{equation}
for every sufficiently large integer $k$ is strongly complete. The exact value of $M_2^{\ast}$ is connected with a question of Graham \cite{Graham71}, also recorded in \cite[p. 58]{ErdosGraham80}, which asks whether 
\begin{equation}\label{eq:A_{alpha,beta}}
A_{\alpha,\beta}:=\{\lfloor 2^k\alpha\rfloor,\lfloor 2^k\beta\rfloor:k\ge0\}\setminus\{0\}    
\end{equation}
is complete with $\alpha,\beta>0$ such that $\alpha/\beta\notin\Q$; see also Bloom's Problem \#354 \cite{Bloom354}. For any $\alpha,\beta\in\R\setminus\{0\}$, we write $\alpha\sim\beta$ if $\alpha$ and $\beta$ are \textit{dyadically equivalent}, meaning $\alpha/\beta=2^n$ for some $n\in\Z$. We call $\alpha$ a \textit{dyadic rational} if $\alpha\sim n$ for some $n\in\Z\setminus\{0\}$. With small modifications, Hegyv\'ari's argument \cite{Hegyvari89} shows that for any pairwise dyadically nonequivalent $\alpha,\beta,\gamma>0$, the set
\[A_{\alpha,\beta,\gamma}:=\{\lfloor 2^n\alpha\rfloor,\lfloor 2^n\beta\rfloor,\lfloor 2^n\gamma\rfloor:n\ge0\}\setminus\{0\}\]
is strongly complete if and only if at least one of $\alpha,\beta,\gamma>0$ is not a dyadic rational. In the same spirit, he conjectured that $A_{\alpha,\beta}$ is complete if $\alpha\not\sim\beta$ and at least one of $\alpha,\beta$ is not a dyadic rational and confirmed the case where exactly one of $\alpha,\beta$ is a dyadic rational. It turns out that the optimal scenario $M_2^{\ast}=2$ would imply the full Hegyv\'ari's conjecture. We include a short proof of this in Remark \ref{rmk:M_2*}. More recently, Hegyv\'ari \cite{Hegyvari24} proved that 
\[
 \FS_0(\lfloor 2^k\alpha\rfloor)+\FS_0(\lfloor 2^k\alpha\rfloor)=\Z_{\ge0}
\]
for every $\alpha\in(0,2)$, where $\FS_0(S):=\FS(S)\cup\{0\}$.

% In Section \ref{S:3dyadic-rays}, we prove the following theorem which asserts that the triple-dyadic-ray set
% \[A_{\alpha,\beta,\gamma}:=\{\lfloor 2^n\alpha\rfloor,\lfloor 2^n\beta\rfloor,\lfloor 2^n\gamma\rfloor:n\ge0\}\setminus\{0\}\]
% is strongly complete under the Hegyv\'{a}ri-type condition.

% \begin{thm}\label{thm:Hegyvari-3}
% Let $\alpha,\beta,\gamma>0$ be pairwise dyadically nonequivalent. Then $A_{\alpha,\beta,\gamma}$ is strongly complete if and only if at least one of $\alpha,\beta,\gamma$ is not a dyadic rational. 
% \end{thm}

% When one of the three parameters $\alpha,\beta,\gamma$ is a dyadic rational and one of them is not, the conclusion of Theorem \ref{thm:Hegyvari-3} already follows from Hegyv\'{a}ri’s mixed two-ray theorem [19]; see also [20, p. 116]. The new content of Theorem \ref{thm:Hegyvari-3} is therefore the case in which none of $\alpha,\beta,\gamma$ is a dyadic rational.

As an application of Theorem \ref{thm:main}, our next result, which will be proved in Section \ref{S:Pf-main} as well, shows that if $A\subseteq\N$ satisfies the Erd\H{o}s conditions \eqref{eq:erdos-growth} and \eqref{eq:L1-divergence}, then it admits a partition into countably many strongly complete sets with prescribed local growth rates. This may be compared with the results of Erdős and Nathanson in \cite{ErdosNathanson88} on partitioning a set into countably many asymptotic bases of a fixed order and with the recent work of Conlon, Fox, and Pham \cite{ConlonFoxPham21} on Ramsey-completeness and coloring results.

\begin{thm}\label{thm:complete-decomposition}
Fix $\rho>1$, let $A\subseteq\N$, and put
\[
 X_k:=A\cap(\rho^k,\rho^{k+1}] \qquad\text{and}\qquad m_k:=|X_k|
\]
for every $k\ge0$. Then the following are equivalent.
\begin{enumerate}
 \item\label{item:erdos-properties}
 $A$ satisfies \eqref{eq:L1-divergence} and $m_k\to\infty$ as $k\to\infty$.
 \item\label{item:every-vector}
 For every sequence $\{\pi_j\}_{j\ge1}$ of positive numbers with $\sum_j\pi_j=1$, there is a partition
 \begin{equation}\label{eq:complete-partition}
 A=\bigcup_{j\ge1}A_j    
 \end{equation}
 into strongly complete sets such that for every $j\ge1$ and $k\ge0$,
 \begin{equation}\label{eq:A_j-growth}
   \big||A_j\cap X_k|-\pi_jm_k\big|<1.
 \end{equation}
 \item\label{item:some-vector}
 There are a sequence $\{\pi_j\}_{j\ge1}$ of positive numbers with $\sum_j\pi_j=1$ and a partition of $A$ into strongly complete sets of the form \eqref{eq:complete-partition} such that \eqref{eq:A_j-growth} holds.
\end{enumerate}
\end{thm}

Another delicate problem of Graham \cite{Graham71}, which was later repeated by Erd\H{o}s and Graham \cite[p. 57]{ErdosGraham80} (see also Bloom's Problem \#349 \cite{Bloom349}), asks for all values of $t,\alpha>0$ such that 
\[
S(t,\alpha):=\{\lfloor t\alpha^n\rfloor:n\in\N\}\setminus\{0\}
\] 
is complete. Graham \cite[Theorem 6]{Graham64b} provided a full characterization of those pairs $(t,\alpha)\in(0,1)\times(1,2)$ for which $S(t,\alpha)$ is complete. He also showed \cite[Theorem 2]{Graham64b} that $S(t,\alpha)$ is complete for $(t,\alpha)\in(0,1)\times(1,\sqrt[3]{5}]$ and that it is incomplete when $t>1$ and $\alpha\ge\max\{2/t,\varphi\}$, where $\varphi=(1+\sqrt{5})/2$ is the golden ratio. In addition, he conjectured that $S(t,\alpha)$ is complete for every $(t,\alpha)\in(0,\infty)\times(1,\varphi)$, but Dubickas \cite{Dubickas06} has shown that if $\alpha_0=1.2527759\dots$ is the larger of the two real roots of the polynomial
\[
P(x)=x^{18}-x^{12}-x^{11}-x^{10}-x^{9}-x^{8}-x^{7}-x^{6}+1,
\]
then there is some $t>0$ such that every member of $S(t,\alpha_0)$ is even, hence disproving Graham's conjecture. More recently, van Doorn \cite{Doorn26} closed the case $\alpha\ge\varphi$ and proved several partial results concerning the complementary range $1<\alpha<\varphi$. 

Closely related to this problem, our next result concerns strong completeness of polynomially perturbed single-ray sets of the form
\[\N\cap\{\lfloor t\alpha^n\rfloor+P_i(n):n\in\N,1\le i\le k\},\]
where $P_1,\ldots,P_k$ are pairwise distinct integer-valued polynomials (with rational coefficients). A special case of it asserts that 
\[
 \N\cap\{\lfloor t\alpha^n\rfloor,\lfloor t\alpha^n\rfloor+P(n):n\in\N\}
\]
is strongly complete for any $t>0$, $\alpha\in(0,2)$, and
primitive integer-valued polynomial $P$. To state our result in a clean manner, we write $\operatorname{Int}(\Z)$ for the set of integer-valued polynomials. For every nonempty $S\subseteq\Z$ containing a nonzero element, define
\[\gcd(S):=\max\{d\in\N: d\mid s\text{ for every }s\in S\}.\]

\begin{thm}\label{thm:poly-talpha^n}
Let $k\ge 2$, $t>0$, and $1<\alpha<\alpha_k$, where 
\[
 \alpha_k:=\frac{k-1+\sqrt{(k-1)^2+8}}{2}
\]
Let
$\Pp=\{P_1,\ldots,P_k\}\subseteq\operatorname{Int}(\Z)$, and put 
\[
 A_{t,\alpha,\Pp}:=\N\cap\{\lfloor t\alpha^n\rfloor+P(n):n\in\N,P\in\Pp\}
\]
and 
\[
 d_{\Pp}:=\gcd\left(\{P_i(n)-P_j(n):n\in\Z,1\le i<j\le k\}\right).
\]
Then the following are equivalent:
\begin{enumerate}
 \item $A_{t,\alpha,\Pp}$ is strongly complete;
 \item \eqref{eq:L1-divergence} holds for $A=A_{t,\alpha,\Pp}$;   
 \item for every prime $p\mid d_{\Pp}$, $p\nmid\lfloor t\alpha^n\rfloor+P_1(n)$ for infinitely many $n$.
\end{enumerate}
In particular, $A_{t,\alpha,\Pp}$ is strongly complete if $d_{\Pp}=1$.
\end{thm}

Theorem \ref{thm:poly-talpha^n} will be proved in Section \ref{S:poly-talpha^n}. The key tool (Lemma \ref{lem:paired-complete}) used to prove this theorem also yields an improvement of Theorem \ref{thm:main} when the elements of $A$ have bounded gaps infinitely often; see Corollary \ref{cor:3-bounded-gaps}. The choice of $P_1$ in condition (3) is immaterial, since $p\mid d_{\Pp}$ implies the same divisibility for $\lfloor t\alpha^n\rfloor+P_i(n)$ for every $1\le i\le k$. Moreover, it can be shown that condition (2) may be replaced by
\[
\sum_{a\in A_{t,\alpha,\Pp}}\|a\theta\|^q=\infty,
  \quad\forall\theta\in\T\setminus\{0\},
\]
for some (or all) $q>0$; see Remark \ref{rmk:H_q(A_{t,alpha,P})}.

% We expect our results to find many interesting applications. The following corollary provides such an example. For every nonempty $S\subseteq\Z$ containing a nonzero element, define
% \[\gcd(S):=\max\{d\in\N: d\mid s\text{ for every }s\in S\}.\]
% This notation will be used throughout the paper.

% \begin{cor}\label{cor:primitive}
% Fix $\rho>1$ and $k_0\in\N$. Let $S\subseteq\Z$ be finite and satisfy $|S|\ge M_{\rho}$ and $\gcd(S-S)=1$, where $M_{\rho}$ is defined as in \eqref{eq:constants}. For every $k\ge k_0$, let $b_k$ be an integer such that $b_k+S\subseteq(\rho^k,\rho^{k+1}]$. Then
% \[
%   A=\bigcup_{k\ge k_0}(b_k+S)
% \]
% is strongly complete.
% \end{cor}
% For instance, if $b_k\in(2^k,2^{k+1}-4]\cap\N$ for every $k\ge k_0$, then
% \[
%   A=\bigcup_{k\ge k_0}\{b_k,b_k+1,\ldots,b_k+5\}
% \]
% is strongly complete, regardless of how the $b_k$ are chosen. Such sets are logarithmically sparse:
% \[\big|A\cap[1,x]\big|=\frac{6}{\log 2}\log x+O(1).\]
% In particular, they have bounded share in each dyadic interval, so they do
% not satisfy Cassels's growth condition.

Finally, we conclude with an application of our machinery to a conjecture of Burr, Erdős, Graham, and Li \cite{BurrErdosGrahamLi96}; see also Erdős Problem 124 \cite{Bloom124}. A slightly stronger set-theoretic formulation of the conjecture appearing in \cite{BergelsonSimmons17} states that if $D\subseteq\N_{\ge2}$ is nonempty and has no two elements which are powers of the same integer, then the set
\[\Pp(D) :=\bigcup_{d\in D}\{d^n: n\in\N\}\]
is strongly complete if and only if $\gcd(D)=1$ and 
\begin{equation}\label{eq:sum-1/(d-1)}
\sum_{d\in D}\frac{1}{d-1}\ge1.   
\end{equation}
It is evident that completeness of $\Pp(D)$ forces $\gcd(D)=1$. Pomerance observed that \eqref{eq:sum-1/(d-1)} is also necessary when $D$ is finite\footnote{Tao has sketched a reconstruction of Pomerance's argument in the comments on \cite{Bloom124}. In greater detail, if $d_1<\cdots<d_k$ are the elements of $D$, where $k\ge2$, and if the sum in \eqref{eq:sum-1/(d-1)} equals $\alpha$, which may be assumed to lie in $(1/2,1)$ without loss of generality, then given $\epsilon>0$, Dirichlet's simultaneous approximation theorem applied to 
\[\frac{\log d_1}{\log d_2},\ldots,\frac{\log d_1}{\log d_k}\notin\Q,\]
yields infinitely many $(m_1,..,m_k)\in\N^k$ such that $|m_i\log d_i-m_jd_j|<2\epsilon$ for $1\le i<j\le k$. For any $\beta\in(1,1/\alpha)$, take $\epsilon$ sufficiently small in terms of $\beta,d_k$, and set $N:=\min_{1\le i\le k}d_i^{m_i}$, which can be arbitrarily large, so that $d_1^{m_1},\ldots,d_k^{m_k}\in[N,\beta N)$. Since each $m_i$ is the greatest integer such that $d_i^{m_i-1}<N$, the sum of all elements of $\Pp(D)\cap[1,N)$ is easily seen to be less than $\alpha\beta N<N$, so $N\notin\FS(\Pp(D))$. Consequently, $\Pp(D)$ is incomplete.}, while Melfi \cite{Melfi04} constructed, for every $\epsilon>0$, an infinite set $D$ for which the sum in \eqref{eq:sum-1/(d-1)} is less than $\epsilon$ while 
\[\Pp_k(D) :=\bigcup_{d\in D}\{d^n: n\ge k\}\]
is complete for every $k\in\N$. Burr, Erdős, Graham, and Li \cite[Theorem 1]{BurrErdosGrahamLi96} proved that $\Pp(D)$ is strongly complete if $\gcd(D)=1$ and $D$ has positive upper density. Bergelson and Simmons \cite[Theorem 1.23]{BergelsonSimmons17} generalized this result, showing that if $D$ admits a partition $D=D_1\cup D_2\cup D_3\cup C$ with $\gcd(C)=1$ and \eqref{eq:sum-1/(d-1)} holds with $D_i$ in place of $D$ for every $i=1,2,3$, then $\Pp(D)$ is strongly complete. In the final section, we obtain the following improvement with one fewer $D_i$ needed. 

\begin{thm}\label{thm:BS-3-sets}
Let $D\subseteq\N_{\ge2}$ be a nonempty finite subset with no two elements being powers of the same integer. Suppose that $D=D_1\cup D_2\cup C$ is a partition of $D$ into nonempty subsets such that $\gcd(C)=1$ and
\begin{equation}\label{eq:sum-D_i}
 \sum_{d\in D_i}\frac1{d-1}\ge1
\end{equation}
for every $i=1,2$. Then $\Pp(D)$ is strongly complete.
\end{thm}

\section{A three-component strong-completeness criterion}\label{S:strong-3set}
Adopting further terminology in \cite{BergelsonSimmons17}, we say that an infinite set $S\subseteq\N$ is \emph{syndetic} if it has bounded gaps, in the sense that there is some $k\in \N$ such that
\[\N\subseteq\bigcup_{j=0}^{k}(S-j),\]
and is \emph{thick} if it contains arbitrarily long intervals. It is an easy fact \cite[Lemma 1]{Graham64a} that if $S$ is syndetic and $T$ is thick, then $S+T$ is \textit{cofinite}, meaning that it covers all but finitely many $n\in\N$. 

For any $S\subseteq\N$, we define
\[\Delta(S):=\sup_{s\in S}\Bigg(s-\sum_{\substack{t\in S\\t<s}}t\Bigg).\]
Observe that $\FS(S)$ is not syndetic if $\Delta(S)=\infty$. Indeed, if $\Delta(S)=\infty$, then $s-\sum_{\substack{t\in S\\t<s}}t$ is bigger than any given number for infinitely many $s\in S$. For every such $s$ there is no subset sum in $\FS(S)$ lying strictly between $\sum_{\substack{t\in S\\t<s}}t$ and $s$. Thus $\FS(S)$ has unbounded gaps. 

We record the following lemma of Burr and Erd\H{o}s
\cite[Lemma 3.2]{BurrErdos81}; see also
\cite[Lemma 2.11]{BergelsonSimmons17}.
% \begin{lem}\label{lem:burr-erdos3.23.1}
% Every $S\subseteq\N$ containing two disjoint subsets $T_1,T_2$, where $T_1$ is syndetic and $T_2$ is thick, must be complete.
% \end{lem}

\begin{lem}\label{lem:burr-erdos3.2}
Let $S\subseteq\N$ be infinite. If $\Delta(S)<\infty$, then $\FS(S)$ is syndetic.
\end{lem}

Thus for an infinite set $S\subseteq\N$, $\FS(S)$ is syndetic if and only if $\Delta(S)<\infty$. It is also easy to see that if $T\subseteq \N$ is any finite subset, then
\begin{equation}\label{eq:Delta(S-T)}
  \Delta(S\setminus T)\le\Delta(S)+\sum_{t\in T\cap S}t.
\end{equation}
In particular, if $\FS(S)$ is syndetic and $T$ is finite, then $\FS(S\setminus T)$ remains syndetic.

Bergelson and Simmons \cite[Theorem 2.1]{BergelsonSimmons17} showed that one can establish completeness of a set $A$ by checking certain conditions for individual components of a partition of $A$. More precisely, suppose that $A=B_1\cup B_2\cup B_3\cup C$ is a set partition of $A$, with $B_1,B_2,B_3,C$ infinite, such that:
\begin{enumerate}[label=(\Roman*)]
    \item for all $1\le i\le 3$, $\Delta(B_i)<\infty$;
    \item for every irrational $\theta\in\R$,
\begin{equation}\label{eq:C-irrational-L1divergence}
  \sum_{c\in C}\|c\theta\|=\infty;
\end{equation}
    \item for every $q\in\N$,
    \begin{equation}\label{eq:C-modular}
  \FS(C)+q\Z=\Z.
\end{equation}
\end{enumerate}
Then $A$ is complete. Compared with \eqref{eq:L1-divergence}, \eqref{eq:C-irrational-L1divergence} only requires divergence of the series at irrational $\theta$, but the modular-completeness condition \eqref{eq:C-modular} compensates it with the local modular information $\FS(C)\pmod* q =\Z/q\Z$. It is easy to prove that 
\begin{equation}\label{eq:rational-theta}
\sum_{c\in C}\|c\theta\|=\infty,\quad\forall\theta\in\Q\setminus\Z, 
\end{equation}
if and only if 
\begin{equation}\label{eq:gcd-tail}
\gcd\big(C\cap[N,\infty)\big)=1,\quad \forall N\in\N.
\end{equation}
Indeed, if $\theta=r/q\in\Q\setminus\Z$ is written in lowest terms, then  $\|c\theta\|=0$ if and only if $q\mid c$, and every nonzero value of $\|c\theta\|$ is at least $1/q$. Consequently, the series diverges precisely when infinitely many $c\in C$ are not divisible by $q$. Requiring this for every $q\ge2$ is equivalent to \eqref{eq:gcd-tail}. Furthermore, it can be shown that \eqref{eq:C-modular} is a consequence of \eqref{eq:rational-theta}, or equivalently, of \eqref{eq:gcd-tail}. Although this implication is already contained in \cite[Lemma 3]{Graham64a} and \cite[Remarks 2.8--2.10]{BergelsonSimmons17}, we include a slightly more general formulation and a short direct proof of it in Appendix; see Proposition \ref{prop:modular}.

The theorem of Bergelson and Simmons \cite[Theorem 2.1]{BergelsonSimmons17} turns out to be sufficient for our approach to yield Theorem \ref{thm:main} with $3u_{\rho}+\max\{u_{\rho},2\}$ in place of $M_{\rho}$. In order to reduce this constant, we need to build a partition of $A$ with fewer components.

Given any $S\subseteq \Z$ and $q>0$, set
\[H_q(S):=\left\{\theta\in\T:\sum_{s\in S}\|s\theta\|^q<\infty\right\},\]
which we call the $H_q$-\textit{spectrum} of $S$. The conditions \eqref{eq:L1-divergence} and \eqref{eq:L2-divergence} can be rephrased as $H_1(A)=\{0\}$ and $H_2(A)=\{0\}$, respectively. With this definition, our three-component strong-completeness criterion can be summarized as the following theorem.

\begin{thm}\label{thm:strong-3set}
Let $B_1,B_2,C\subseteq\N$ be pairwise disjoint infinite sets with $\Delta(B_1),\Delta(B_2)<\infty$ and $H_1(C)=\{0\}$. Assume further that either $\Delta(C)<\infty$ or $H_2(C)$ is countable. Then $A=B_1\cup B_2\cup C$ is strongly complete.
\end{thm}

As we shall see in Section \ref{S:Pf-main}, the improved constant $M_{\rho}$ arises from a three-component partition as described in Theorem \ref{thm:strong-3set} with countability of $H_2(C)$ fulfilled. To prove this theorem, we start by examining the proof of \cite[Theorem 2.1]{BergelsonSimmons17} to extract the following thickness result.

% Next, we extract the following useful result by examining the proof of \cite[Theorem 2.1]{BergelsonSimmons17}.

\begin{lem}\label{lem:BS2017}
Let $B_1,B_2,C\subseteq\N$ be pairwise disjoint. Suppose that
$\FS(B_1)$ and $\FS(B_2)$ are syndetic and that both \eqref{eq:C-irrational-L1divergence} and \eqref{eq:C-modular} hold for $C$. Then there exists a finite set $E\subseteq C$ such that
\[
  \FS(B_1\cup B_2\cup E)
\]
is thick.
\end{lem}

\begin{proof}
We follow the proof of \cite[Theorem 2.1]{BergelsonSimmons17}. It is proved there that one can find
$d\in\N$, $\boldsymbol\theta\in\T^d$, a nonempty open set
$U\subseteq\T^d$, and a thick set $J\subseteq\N$ such that
\begin{equation}\label{eq:piecewise-bohr}
  \FS(B_1\cup B_2)\supseteq\FS(B_1)+\FS(B_2)
  \supseteq
  J\cap\{n\in\N:n\boldsymbol\theta\in U\}
  \ne\emptyset.
\end{equation}
Define
% \begin{equation}\label{eq:G}
%  G:=\bigcap_{N\in\N}
%   \overline{
%     \{n\boldsymbol\theta:
%       n\in\FS(C\cap[N,\infty))\}
%   }
%   \subseteq\T^d.
% \end{equation}
\[G:=\bigcap_{N\in\N}
  \overline{
    \{n\boldsymbol\theta:
      n\in\FS(C\cap[N,\infty))\}
  }
  \subseteq\T^d.\]
Claims 2.13 and 2.14 in the proof of \cite[Theorem 2.1]{BergelsonSimmons17} show that $G$ is a compact subgroup of $\T^d$ and that there exists $q\ge1$ such that $q\boldsymbol\theta\in G$. These are derived only from \eqref{eq:C-irrational-L1divergence} and the assumption that $\FS(B_1)$ and $\FS(B_2)$ are syndetic. Let
\[
  H:=G+\{0,1,\ldots,q-1\}\boldsymbol\theta\supseteq\N\boldsymbol\theta.
\]
Then $H$ is a compact subgroup of $\T^d$. By \eqref{eq:C-modular} we have
\[ \{n\boldsymbol\theta:n\in\FS(C)\}+G=H.\]
Continuing the proof of \cite[Theorem 2.1]{BergelsonSimmons17} yields a finite set $D\subseteq\FS(C)$ such that
\begin{equation}\label{eq:D}
H\subseteq\bigcup_{n\in D}(n\boldsymbol\theta+U).  
\end{equation}
Let 
\[T:=\{n>\max D:n-m\in J\text{ for every }0\le m\le \max D\}.\]
As $J$ is thick, we know that $T$ is also thick. So far, the argument has been essentially the same as in \cite[Theorem 2.1]{BergelsonSimmons17}.

For each $m\in D$, choose a representation of $m$ as a nonempty sum of distinct elements of $C$, and let $E\subseteq C$ be the set of all the summands that appear in these representations. Then $E$ is finite and $D\subseteq\FS(E)$. Fix $n\in T$.  Since
$n\boldsymbol\theta\in H$, \eqref{eq:D} produces some
$m\in D$ such that $(n-m)\boldsymbol\theta\in U$.  Since
$0\le m\le \max D$, we also have $n-m\in J$.  Thus
\eqref{eq:piecewise-bohr} gives
\[
  n-m\in\FS(B_1)+\FS(B_2),
\]
while $m\in\FS(E)$.  The sets $B_1,B_2,E$ are pairwise disjoint, so
\[n\in\FS(B_1)+\FS(B_2)+\FS(E)\subseteq\FS(B_1\cup B_2\cup E).\]  
Hence, $\FS(B_1\cup B_2\cup E)$ contains the thick set $T$ and is therefore also thick.
\end{proof}

Lemma \ref{lem:BS2017} alone yields immediately the following completeness criterion based on a three-component partition of $A$.
\begin{cor}\label{cor:3set}
Let $B_1,B_2,C\subseteq\N$ be pairwise disjoint infinite sets.  Suppose that 
\[\Delta(B_1),\ \Delta(B_2),\ \Delta(C)<\infty,\] 
and that both \eqref{eq:C-irrational-L1divergence} and \eqref{eq:C-modular} hold.
Then $A=B_1\cup B_2\cup C$ is complete.
\end{cor}

\begin{proof}
By Lemma \ref{lem:burr-erdos3.2}, the sets $\FS(B_1)$ and $\FS(B_2)$ are
syndetic. Lemma \ref{lem:BS2017} therefore supplies a finite set
$E\subseteq C$ such that $\FS(B_1\cup B_2\cup E)$ is thick. By \eqref{eq:Delta(S-T)} and Lemma \ref{lem:burr-erdos3.2}, the set $\FS(C\setminus E)$ is syndetic. So $\FS(B_1\cup B_2\cup E)+\FS(C\setminus E)$ is cofinite. Moreover, disjointness implies
\[
  \FS(B_1\cup B_2\cup E)+\FS(C\setminus E)
  \subseteq
  \FS(A).
\]
Thus $A$ is complete.
\end{proof}

Together with Proposition \ref{prop:modular} in Appendix, Corollary \ref{cor:3set} enables us to handle the case $\Delta(C)<\infty$ of Theorem \ref{thm:strong-3set}. We now proceed to develop the necessary tools for the treatment of the case when countability of $H_2(C)$ is assumed instead. To this end, we recall some definitions and facts.

For any $X\subseteq\Z$, we write
\[
 d^*(X):=\limsup_{N\to\infty}\sup_{m\in\Z}
 \frac{|X\cap[m+1,m+N]|}{N}
\]
for its upper Banach density. It is easy to see that $d^*(X)>0$ for any syndetic set $X\subseteq\N$.

For a probability measure $\mu$ on $\Z$, denote its support by
\[
 \operatorname{supp}\mu:=\{m\in\Z:\mu(\{m\})>0\},
\]
and define its Fourier transform by
\[\widehat\mu(\theta):=\int_{\Z}\e^{2\pi i m\theta}\,d\mu(m)=\sum_{m\in\Z}\mu(m)\e^{2\pi i m\theta},\quad\forall\theta\in\T.\]
Given probability measures $\mu_1$ and $\mu_2$ on $\Z$, the convolution $\mu_1*\mu_2$ is a probability measure on $\Z$ defined by 
\[(\mu_1*\mu_2)(n):=\sum_{m\in\Z}\mu_1(m)\mu_2(n-m),\quad\forall n\in\Z.\]
A standard fact in Fourier analysis states that $\widehat{\mu_1*\mu_2}=\widehat\mu_1\cdot\widehat\mu_2$.

Following Hartman \cite{Hartman68} and, for the measure formulation, Griesmer \cite[Definition 10.1]{Griesmer26}, a sequence $\{\mu_n\}_{n\ge1}$ of probability measures on $\Z$ is called \textit{Hartman uniformly distributed} if 
\[
 \lim_{n\to\infty}\widehat\mu_n(\theta)=0
\]
for every $\theta\in\T\setminus\{0\}$. The following result, due to Griesmer \cite{Griesmer12}, provides a recipe for producing a thick sumset.

\begin{lem}\label{Lem:thick-sumset}
Let $X,Y\subseteq\Z$ with $d^*(X)>0$. Suppose that there is a
Hartman uniformly distributed sequence of probability measures $\{\mu_n\}_{n\ge1}$ on $\Z$, each supported on $Y$. Then $X+Y$ is thick.
\end{lem}

\begin{proof}
This follows from \cite[Theorem 1.4(2)]{Griesmer12} applied with $A=Y$, $B=X$, and $\nu_n=\mu_n$. The assumption that each $\mu_n$ is supported on $Y$ implies that $\mu_n(Y)=1$ for every $n$, and hence
\[
 d_\mu(Y):=\limsup_{n\to\infty}\mu_n(Y)=1.
\]
Together with $d^*(X)>0$, this verifies all the conditions required by \cite[Theorem 1.4(2)]{Griesmer12}, which then tells us that $A+B$ is thick.
\end{proof}

The next lemma allows us to construct Hartman uniformly distributed sequences of probability measures with certain support properties needed for our application of Lemma \ref{Lem:thick-sumset}.

\begin{lem}\label{lem:Hartman-H_1-H_2}
Let $D\subseteq\N$ and suppose that $H_2(D)$ is countable. Then $H_1(D)=\{0\}$ if and only if for every $N\in\N$ there is a Hartman uniformly distributed sequence $\{\mu_r\}_{r\ge1}$ of finitely supported probability measures on $\Z$ such that $\operatorname{supp}\mu_r\subseteq\FS(D\cap[N,\infty))$ for every $r\ge1$ and $\min\operatorname{supp}\mu_r\to\infty$ as $r\to\infty$.
\end{lem}

\begin{proof}
Note that countability of $H_2(D)$ implies that $D$ is infinite. For the sufficiency part, let $\{\mu_r\}_{r\ge1}$ be a Hartman uniformly distributed sequence of finitely supported probability measures with the asserted properties. Fix
$\theta\in H_1(D)$ and choose $N\in\N$ so large that
\[
 \sum_{d\in D\cap[N,\infty)}\|d\theta\|<\frac{1}{6}.
\]
Every $m\in\FS(D\cap[N,\infty))$ then satisfies
$\|m\theta\|<1/6$, and consequently
\[\text{Re}\big(\widehat\mu_r(\theta)\big)=\sum_{m\in\Z}\mu_r(m)\cos\big(2\pi m\theta\big)>\frac{1}{2}\]
for all $r\ge 1$, thanks to $\operatorname{supp}\mu_r\subseteq\FS(D\cap[N,\infty))$. Since $\{\mu_r\}_{r\ge1}$ is Hartman uniformly distributed, this forces $\theta=0$. Hence, $H_1(D)=\{0\}$.

To prove the necessity part, suppose that $H_1(D)=\{0\}$. Fix $N\in\N$ and enumerate elements of $D$ as a strictly increasing sequence $D=\{d_n\}_{n\ge1}$. Since $H_2(D)\setminus\{0\}$ is countable, we may enumerate it as
\[
 H_2(D)\setminus\{0\}=\{\theta_j:1\le j<k\},
\]
where $k\in\N\cup\{\infty\}$. For every $1\le j<k$, choose $x_n^{(j)}\in[-1/2,1/2)$ such that $x_n^{(j)}=d_n\theta_j\pmod* 1$. Then $\big|x_n^{(j)}\big|=\|d_n\theta_j\|$.
Since $\theta_j\in H_2(D)$, the series
$\sum_n\big|x_n^{(j)}\big|^2$ converges, and in particular
$x_n^{(j)}\to0$ as $n\to\infty$. On the other hand, $\theta_j\ne0$ and
$H_1(D)=\{0\}$ imply $\sum_n\big|x_n^{(j)}\big|=\infty$.

We claim that given $1\le j<k$, $K>0$ and $\epsilon\in(0,1/2)$, there is a nonempty finite set $I\subseteq\N_{>K}$ such that
\begin{equation}\label{eq:sum-half-approx}
 \left\|\sum_{n\in I}d_n\theta_j-\frac12\right\|
 <\epsilon.
\end{equation}
Indeed, let $K_1>K$ be an integer such that $\big|x_n^{(j)}\big|<\epsilon$ for all $n\ge K_1$. Then at least one of the two subseries
\[
 \sum_{\substack{n>K_1\\x_n^{(j)}>0}}x_n^{(j)},
 \qquad
 \sum_{\substack{n>K_1\\x_n^{(j)}<0}}(-x_n^{(j)})
\]
diverges. It suffices to take the convergent one, choose the unique integer $K_2>K_1$ for which its partial sum over $K_1\le n\le K_2$ lies in $[1/2,1/2+\epsilon)$, and declare $I=[K_1,K_2]\cap\N$.

We now construct the required probability measures $\{\mu_r\}_{r\ge1}$ using \eqref{eq:sum-half-approx}. Fix arbitrary sequences $\{\epsilon_r\}_{r\ge1}\subseteq(0,1/4)$ and $\{L_r\}_{r\ge1}\subseteq \N_{\ge N}$ such that $\epsilon_r\to0$ and $L_r\to\infty$ as $r\to\infty$. For each $r\ge1$, let $k_r:=\min\{k,r\}$. Define the probability measure
\[
\nu_r:=\mathop{*}_{n=N}^{L_r}\frac{\delta_0+\delta_{d_n}}2,
\]
where $\delta_m$ denotes the Dirac measure at $m$ for every $m\in\Z$. Equivalently, $\nu_r$ is the distribution of $\sum_{n=N}^{L_r}d_n\mathbf{X}_n$, where the $\mathbf{X}_n$ are independent Bernoulli random variables taking the values $0$ and $1$ with probability $1/2$ each. Successively for $j=1,2,\ldots,k_r-1$, we apply \eqref{eq:sum-half-approx} with 
$
 K=\max(\{L_r\}\cup\bigcup_{1\le\ell<j}I_{r,\ell})
$ and $\epsilon=\epsilon_r$ to obtain pairwise disjoint finite sets $I_{r,j}\subseteq\N_{>L_r}$, such that for each $1\le j<k_r$,
\begin{equation}\label{eq:s_{r,j}-1/2-approx}
\left\|s_{r,j}\theta_j-\frac12\right\|
 <\epsilon_r,    
\end{equation}
where
\[
 s_{r,j}:=\sum_{n\in I_{r,j}}d_n.
\]
Choose an arbitrary integer
\[
 m_r>\max\left(\{L_r\}\cup
       \mathop{\bigcup}_{j=1}^{k_r-1}I_{r,j}\right)
\]
and define
\[
 \mu_r:=\delta_{d_{m_r}}*\nu_r*
        \mathop{*}_{j=1}^{k_r-1}
        \frac{\delta_0+\delta_{s_{r,j}}}{2}.
\]
Equivalently, $\mu_r$ is the distribution of \[d_{m_r}+\sum_{n=N}^{L_r}d_n\mathbf{X}_n+\sum_{j=1}^{k_r-1}\mathbf{Y}_j\sum_{n\in I_{r,j}}d_n,\] 
where the $\mathbf{X}_n,\mathbf{Y}_j$ are independent Bernoulli random variables taking the values $0$ and $1$ with probability $1/2$ each. It is clear that $\mu_r$ is finitely supported with $\operatorname{supp}\mu_r\subseteq\FS(\{d_n:n\ge N\})$.
Moreover, 
\[
 \min\operatorname{supp}\mu_r\ge d_{m_r}>d_{L_r}
 \to\infty
\]
as $r\to\infty$.

It remains to verify Hartman uniformity.  For an arbitrary $\theta\in\T$, Fourier multiplicativity leads to
\begin{align}
 \widehat\mu_r(\theta)
 &=\e^{2\pi i d_{m_r}\theta}
   \prod_{n=N}^{L_r}\frac{1+\e^{2\pi i d_n\theta}}2
   \prod_{j=1}^{k_r-1}\frac{1+\e^{2\pi i s_{r,j}\theta}}2.
 \label{eq:Fourier-mu_r}
\end{align}
First suppose that $\theta\notin H_2(D)$. Then
\[\big|\widehat\mu_r(\theta)\big|\le\prod_{n=N}^{L_r}|\cos(\pi d_n\theta)|
\le\exp\left(-\pi\sum_{n=N}^{L_r}\| d_n\theta\|^2\right)
 \to0
\]
as $r\to\infty$, where we have used the inequality 
\[|\cos(\pi t)|\le 1-\pi t^2\le \e^{-\pi t^2}\] 
for $t\in[-1/2,1/2]$, together with $\sum_n\| d_n\theta\|^2=\infty$.

Now suppose that $\theta\in H_2(D)\setminus\{0\}$, say $\theta=\theta_{j_0}$ with $0\le j_0< k$. For every $r>j_0$, we have $k_r-1\ge j_0$, so that \eqref{eq:s_{r,j}-1/2-approx} and \eqref{eq:Fourier-mu_r} yield
\[
 \big|\widehat\mu_r(\theta)\big|\le|\cos(\pi s_{r,j_0}\theta_{j_0})|
 \le\pi\epsilon_r.
\]
In particular, $\widehat\mu_r(\theta)\to0$ as $r\to\infty$.

This shows that $\{\mu_r\}_{r\ge1}$ is Hartman uniformly distributed. The proof is therefore complete.
\end{proof}

We are now ready to prove Theorem \ref{thm:strong-3set} by packaging the preceding results in this section.

\begin{proof}[Proof of Theorem \ref{thm:strong-3set}]
Let $E\subseteq A$ be finite and put $A'=A\setminus E$, $B_j'=B_j\setminus E$ for $j=1,2$, and $C'=C\setminus E$. By \eqref{eq:Delta(S-T)} we have
\[
  \Delta(B_1'),\ \Delta(B_2')<\infty.
\]
In addition, $H_1(C')=\{0\}$, since $H_1(C)=\{0\}$ and deleting finitely many nonnegative terms does not affect divergence. It suffices to show that $A'$ is complete.

Suppose first that $\Delta(C)<\infty$. Then $\Delta(C')<\infty$ as well. Note that $H_1(C')=\{0\}$ implies \eqref{eq:C-irrational-L1divergence} for $\theta\in\R\setminus\Q$ with $C$ replaced by $C'$, and, together with
Proposition \ref{prop:modular}, gives \eqref{eq:C-modular} with $C'$ in place of $C$. Corollary \ref{cor:3set} applied to 
\[A'=B_1'\cup B_2'\cup C'\] 
shows that $A'$ is complete. 

Now suppose that $H_2(C)$ is countable. Then $H_2(C')$ is also countable. By Lemma \ref{lem:burr-erdos3.2}, $\FS(B_1')$ and $\FS(B_2')$ are syndetic. In particular, $d^*(\FS(B_1'))>0$.
Lemma \ref{lem:Hartman-H_1-H_2} therefore supplies a sequence of Hartman uniformly distributed
probability measures supported on $\FS(C')$, and
Lemma \ref{Lem:thick-sumset} then makes $ \FS(B_1')+\FS(C')$ thick. It follows that $\FS(B_1')+\FS(B_2')+\FS(C')$ is cofinite. Since $B_1',B_2',C'$ are pairwise disjoint, we have 
\[\FS(A')\supseteq\FS(B_1')+\FS(B_2')+\FS(C').\] 
In particular, $\FS(A')$ is cofinite. Hence, $A'$ is complete. 
\end{proof}

\begin{rmk}
The two hypotheses on $C$ in Theorem \ref{thm:strong-3set}, namely $\Delta(C)<\infty$ and countability of $H_2(C)$, characterize different phenomena and do not imply each other even under $H_1(C)=\{0\}$. Using Lemma \ref{lem:bounded-ratio} in the next section, it is easy to see that $C_1:=\{3^n+n:n\in\N\}$ has trivial $H_1$-spectrum and countable $H_2$-spectrum but infinite $\Delta$-value. On the other hand, the set 
\[C_2:=\bigcup_{k\ge 1}\big\{n_k,2n_k,\ldots,2^{k+2}n_k\big\},\]
where the sequence $\{n_k\}_{n\ge1}$ is defined recursively via $n_1=1$ and $n_{k+1}=2^{k+1}(2^{k+2}+1)n_k+1$ for $k\ge1$, satisfies $H_1(C_2)=\{0\}$ and $\Delta(C_2)=1$, but $H_2(C_2)$ is uncountable. We leave the details to the interested reader.
\end{rmk}

With Theorem \ref{thm:strong-3set} at hand, the task of proving Theorem \ref{thm:main} reduces to the construction of a partition of $A$ with the specified properties.

\section{Spectral countability and deletion}\label{S:spectral-deletion}
In this section we gather a few technical lemmas needed for the proof of Theorem \ref{thm:main}. Our first lemma below ensures countability of the $H_q$-spectrum of any bounded-ratio sequence. This countability phenomenon is classical and goes back to Eggleston \cite{Eggleston52}. Erd\H{o}s and Taylor \cite{ErdosTaylor57} later obtained more general results. We include a short direct proof adapted to the present formulation.
\begin{lem}\label{lem:bounded-ratio}
Fix $\lambda>1$, and let $\{a_n\}_{n\ge1}$ be an unbounded sequence of positive integers such that
\begin{equation}\label{eq:bounded-ratio}
  a_{n+1}\le\lambda a_n,
  \quad\forall n\ge1.
\end{equation}
Then 
\[
  E(a):=\{\theta\in\T:\|a_n\theta\|\to0\text{ as }n\to\infty\}
\]
is countable. Consequently $H_q(\{a_n:n\ge1\})$ is also countable for every $q>0$.
\end{lem}

\begin{proof}
The second assertion follows from the first, since the terms of a convergent series eventually approach 0, so $H_q(a)\subseteq E(a)$. To prove the first assertion, fix 
\[
  0<\epsilon<\frac{1}{8\lambda},
\]
and define 
\[
  E_{N}(a):=\{\theta\in\T:\|a_n\theta\|\le\epsilon
  \text{ for every }n\ge N\}
\]
for every $N\in\N$. We claim that $E_{N}(a)$ is finite. Suppose that $\theta,\phi\in E_{N}(a)$ are distinct and put $\psi=\|\theta-\phi\|>0$. We show that
\begin{equation}\label{eq:psi}
  \psi\ge\frac{1}{4\lambda a_N}.
\end{equation}
Once this is proved, we know that $|E_{N}(a)|\le 4\lambda a_N$ and that
\[E(a)\subseteq \bigcup_{N=1}^{\infty}E_{N}(a)\]
is countable. So it is sufficient to verify \eqref{eq:psi}.

Assume to the contrary that \eqref{eq:psi} is false. Since $\{a_n\}_{n\ge1}$ is unbounded, we may choose the least $n>N$ such that
\[a_n\psi\ge\frac{1}{4\lambda}.\]
By minimality of $n$ we also have
\[a_{n-1}\psi<\frac{1}{4\lambda}.\]
Thanks to \eqref{eq:bounded-ratio}, these two inequalities combine to give
\[\frac{1}{4\lambda}\le a_n\psi<\frac{1}{4}.\]
Since $a_n\in\N$ and $a_n\psi\in(0,1/4)$, we have $\|a_n(\theta-\phi)\|=\|a_n\psi\|=a_n\psi\ge1/(4\lambda)$. Hence, the triangle inequality gives
\[
\frac{1}{4\lambda}\le\|a_n(\theta-\phi)\|\le\|a_n\theta\|+\|a_n\phi\|\le2\epsilon<\frac{1}{4\lambda},
\]
which is absurd. This verifies \eqref{eq:psi}.
\end{proof}

Next, we need the following deletion lemma which allows one to delete a collection of finite sets of prescribed sizes while retaining the divergence property. It will be applied in the next section with $r_k=r$ for all sufficiently large $k$ and each $f_j(x)$ taking the form of $\|\theta_j x\|$ with some $\theta_j\in\T\setminus\{0\}$.

\begin{lem}\label{lem:countable-deletion}
Let $\{X_k\}_{k\ge1}$ be a sequence of pairwise disjoint finite sets such that $|X_k|>r_k$ for all $k\ge1$, where $r_k\in\Z_{\ge0}$, and set $X=\bigcup_{k\ge1}X_k$. Let $J\subseteq\N$ be countable.  For every $j\in J$, let $f_j:X\to[0,\infty)$ be a function satisfying
\begin{equation}\label{eq:weight-divergence}
  \sum_{k=1}^{\infty}\frac{|X_k|-r_k}{|X_k|}\sum_{x\in X_k}f_j(x)=\infty.
\end{equation}
Then there exist sets $D_k\subseteq X_k$, with $|D_k|=r_k$, such that
\begin{equation}\label{eq:retained-divergence}
  \sum_{k=1}^{\infty}\sum_{x\in X_k\setminus D_k}f_j(x)=\infty,
  \quad\forall j\in J.
\end{equation}
\end{lem}

This result has a simple probabilistic interpretation: For each $k\in\N$, let $D_k\subseteq X_k$ be an $r_k$-element subset chosen uniformly at random. Then the probability that a given $x\in X_k$ is not selected is
\[\PP(x\in X_k\setminus D_k)=1-\frac{r_k}{|X_k|}=\frac{|X_k|-r_k}{|X_k|}.\]
For $j\in J$, define the random variable
\[Y_{j}:=\sum_{k=1}^{\infty}\sum_{x\in X_k\setminus D_k}f_j(x)=\sum_{k=1}^{\infty}\sum_{x\in X_k}f_j(x)\mathbf{1}_{x\notin D_k},\]
and put
\[E:=\bigcap_{j\in J}\{Y_j=\infty\}.\]
Lemma \ref{lem:countable-deletion} asserts that if $\E(Y_j)=\infty$ for every $j\in J$, then $E\ne\emptyset$. 

\begin{proof}[Proof of Lemma \ref{lem:countable-deletion}]
The case $J=\emptyset$ is vacuous. So we suppose that $J\ne\emptyset$. Since $J$ is countable, we can choose a sequence $\{i_s\}_{s\ge 1}$ of elements of $J$, repetitions allowed, such that every $j\in J$ occurs infinitely often. 

We build $D_k$ recursively. Set $K_1=1$. Suppose that at stage $s$ the number $K_s$ has been defined and the sets $D_k~(k<K_s)$ have been chosen. Consider $i_s$ in the sequence. By \eqref{eq:weight-divergence}, we can find $L_s\ge K_s$ such that
\[
   \sum_{k=K_s}^{L_s}\frac{|X_k|-r_k}{|X_k|}\sum_{x\in X_k}f_{i_s}(x)\ge 1.
\]
For each $k\in[K_s,L_s]$, let $D_k$ consist of the $r_k$ elements of $X_k$
with the least $f_{i_s}$-values. Thus
\[\frac{1}{|X_k|-r_k}\sum_{x\in X_k\setminus D_k}f_{i_s}(x)\ge\frac{1}{|X_k|}\sum_{x\in X_k}f_{i_s}(x),\]
from which it follows that
\[\sum_{k=K_s}^{L_s}\sum_{x\in X_k\setminus D_k}f_{i_s}(x)\ge\sum_{k=K_s}^{L_s}\frac{|X_k|-r_k}{|X_k|}\sum_{x\in X_k}f_{i_s}(x)\ge 1.\]
Now set $K_{s+1}=L_s+1$ and repeat this process for $i_{s+1}$. Continuing in this way, the consecutive finite intervals $[K_s,L_s]~(s\in\N)$ form a partition of $[1,\infty)$, and each set $D_k\subseteq X_k$ has $|D_k|=r_k$. Since every $j\in J$ appears infinitely often in the sequence $\{i_s\}_{s\ge 1}$, given every $c>0$ and $j\in J$ there exists $s_0=s_0(c,j)\in\N$ such that $j$ occurs at
least $\lceil c\rceil$ times among $i_1,\ldots,i_{s_0}$ and thus
\[\sum_{k=1}^{L_{s}}\sum_{x\in X_k\setminus D_k}f_{j}(x)\ge c,\quad\forall s\ge s_0.\]
Hence, \eqref{eq:retained-divergence} holds for every $j\in J$.
\end{proof}

Finally, we record the following elementary estimate responsible for the constants in \eqref{eq:constants}. It provides a sufficient local abundance condition for a set to be syndetic.

\begin{lem}\label{lem:Delta(S)}
Let $\rho>1$, $r\in\N$, and $S\subseteq\N$. For each $k\ge0$, let 
\[
m_k:=\big|S\cap(\rho^k,\rho^{k+1}]\big|,
\]
and suppose that
\begin{equation}\label{eq:m_k}
 \sum_{j=0}^{r-1}m_{k+j}\rho^j \ge \rho(\rho^r-1)    
\end{equation}
for every sufficiently large $k$, then $\Delta(S)<\infty$. 
\end{lem}

\begin{proof}
Suppose that \eqref{eq:m_k} holds for every $k\ge k_0$. If $s\in S\cap(\rho^k,\rho^{k+1}]$ with $k>k_0$, then 
\begin{align*}
\sum_{\substack{t\in S\\t<s}}t\ge\sum_{\ell=k_0}^{k-1}m_{\ell}\rho^\ell
&\ge\sum_{1\le \ell\le (k-k_0)/r}\rho^{k-r\ell}\sum_{j=0}^{r-1}m_{k-r\ell+j}\rho^j\\
&\ge\rho(\rho^r-1)\sum_{1\le \ell\le (k-k_0)/r}\rho^{k-r\ell} \\
&=\rho(\rho^r-1)\left(\frac{\rho^{k-r}}{1-\rho^{-r}}+O_{\rho,k_0,r}(1)\right)\\
&=\rho^{k+1}+O_{\rho,k_0,r}(1).
\end{align*}
Hence
\[
  s-\sum_{\substack{t\in S\\t<s}}t
  \le
  \rho^{k+1}-\left(\rho^{k+1}+O_{\rho,k_0,r}(1)\right)
  \ll_{\rho,k_0,r}1.
\]
This proves $\Delta(S)<\infty$.
\end{proof}

\section{Resolution of Erd\H{o}s's conjecture: Proofs of Theorems \ref{thm:main} and \ref{thm:complete-decomposition}}\label{S:Pf-main}
We are now in a position to deliver the proof of Theorem \ref{thm:main} based on our three-component strong-completeness criterion (Theorem \ref{thm:strong-3set}). Before embarking on the proof, we make an additional observation on the growth of $q_A(n)$.

\begin{lem}\label{lem:q_A}
Let $\rho>1$, $h\in\N$, and $A\subseteq \N$. Suppose that $A=B\cup C$ is a partition such that $C$ is strongly complete and 
\[
\big|B\cap(\rho^k,\rho^{k+1}]\big|\ge h
\]
for every sufficiently large $k$. Then
\[
\lim_{n\to\infty} \frac{q_{A\setminus F}(n)}
 {n^{h\log_\rho2}}=\infty
\]
for every finite subset $F\subseteq A$.
\end{lem}

\begin{proof}
Given any finite $F\subseteq A$, there exists $n_0\in\N$ such that $n\in\FS(C\setminus F)$ for every $n\ge n_0$. For each $n\ge 2n_0$, let $B_n\subseteq B\setminus F$ be a finite set with maximal cardinality, such that
\[
\sum_{b\in B_n}b\le \frac{n}{2}.
\]
Then 
\[n-\sum_{d\in D}d\in\FS(C\setminus F)\]
for every $n\ge 2n_0$ and every $D\subseteq B_n$. Since
\[|B_n|=h\log_{\rho}n+O_{\rho,h,F}(1),\]
we conclude that
\[
q_{A\setminus F}(n)\gg_{\rho,h,F} 2^{h\log_{\rho}n}\min_{m\ge n/2}q_{C\setminus F}(m)
\]
for every $n\ge 2n_0$. In view of \eqref{eq:q_A} and the identity $2^{h\log_{\rho}n}=n^{h\log_{\rho}2}$, this proves the asserted growth rate for $q_{A\setminus F}(n)$.
\end{proof}

Now we prove Theorem \ref{thm:main}.

\begin{proof}[Proof of Theorem \ref{thm:main}]
Fix $\rho>1$ and $k_0\in \N$. Suppose that $A\subseteq\N$ satisfies \eqref{eq:L1-divergence} and the inequality $|A\cap(\rho^k,\rho^{k+1}]|\ge M\ge M_\rho$ for all $k\ge k_0$, where $M_\rho$ is defined as in \eqref{eq:constants}. Without loss of generality, we may replace $A$ with $A\setminus[1,\rho^{k_0}]$.

We first prove the base case $M=M_{\rho}$ of the theorem. Let $A_k:=A\cap(\rho^k,\rho^{k+1}]$ for $k\ge k_0$. Then $|A_k|\ge M_\rho$. Divide $\N_{\ge k_0}$ into even and odd classes 
\[E:=\{k\ge k_0: 2\mid k\},\quad O:=\{k\ge k_0: 2\nmid k\},\]
and define
\[
 A_E:=\bigcup_{k\in E}A_k,
 \qquad A_O:=\bigcup_{k\in O}A_k.
\]
Choose an arbitrary element from $A_k$ for each $k\in O$ to form a strictly increasing sequence $S_O$ of bounded ratios
less than $\rho^3$. By Lemma \ref{lem:bounded-ratio}, $H_1(S_O)$ is countable. Since $S_O\subseteq A_O$, we know that $H_1(A_O)$ is countable as well. In view of \eqref{eq:L1-divergence}, we find
\begin{equation}\label{eq:even-divergence}
 \sum_{k\in E}\sum_{a\in A_k}\| a\theta\|=\infty,\quad\forall \theta\in H_1(A_O)\setminus\{0\}.
\end{equation}
Since 
\[
 \frac{|A_k|-(M_\rho-1)}{|A_k|}\ge\frac1{M_\rho},
\]
\eqref{eq:even-divergence} implies
\[
  \sum_{k\in E}\frac{|A_k|-(M_\rho-1)}{|A_k|}\sum_{a\in A_k}\| a\theta\|=\infty,\quad\forall \theta\in H_1(A_O)\setminus\{0\}.
\]
Applying Lemma \ref{lem:countable-deletion} to $X_k=A_k$ for $k\in E$ with $r_k=M_\rho-1$ and $f_{\theta}(a)=\|a\theta\|$ for $\theta\in H_1(A_O)\setminus\{0\}$, we find $D_k\subseteq A_k$ for $k\in E$ with $|D_k|=M_\rho-1<|A_k|$ such that
\begin{equation}\label{eq:C_E-divergence}
   \sum_{c\in C_E}\|c\theta\|=\sum_{k\in E}\sum_{a\in A_k\setminus D_k}\|a\theta\|=\infty,
  \quad\forall \theta\in H_1(A_O)\setminus\{0\},
\end{equation}
where
\[C_E:=\bigcup_{k\in E}(A_k\setminus D_k).\]
By \eqref{eq:C_E-divergence}, we see that
\begin{equation}\label{eq:H_1(C_E)-H_1(A_O)}
 H_1(C_E)\cap H_1(A_O)=\{0\}.
\end{equation}

Now choose an arbitrary element from $A_k\setminus D_k$ for each $k\in E$ to form a strictly increasing sequence $S_E$ of bounded ratios less than $\rho^3$. Again, by Lemma \ref{lem:bounded-ratio}, $H_1(S_E)$ is countable, and so is $H_1(C_E)$. Since \eqref{eq:H_1(C_E)-H_1(A_O)} implies
\[ \sum_{k\in O}\sum_{a\in A_k}\| a\theta_j\|=\infty,\quad\forall \theta\in H_1(C_E)\setminus\{0\},\]
Lemma \ref{lem:countable-deletion} once more applied to $X_k=A_k$ for $k\in O$, with $r_k=M_\rho-1$ and $f_{\theta}(a)=\|a\theta\|$ for $\theta\in H_1(C_E)\setminus\{0\}$, produces $D_k\subseteq A_k$ for $k\in O$ with $|D_k|=M_\rho-1<|A_k|$ such that
\begin{equation}\label{eq:C_O-divergence}
   \sum_{c\in C_O}\|c\theta_j\|=\sum_{k\in O}\sum_{a\in A_k\setminus D_k}\|a\theta\|=\infty,\quad\forall \theta\in H_1(C_E)\setminus\{0\},
\end{equation}
where 
\[C_O:=\bigcup_{k\in O}(A_k\setminus D_k).\]

Now we construct the partition needed for the application of Theorem \ref{thm:strong-3set}, which will allow us to conclude that $A$ is strongly complete. Let
\[
x_{\rho}:=\left\lfloor \frac{M_{\rho}-1}{2}\right\rfloor,\qquad
y_{\rho}:=\left\lceil \frac{M_{\rho}-1}{2}\right\rceil.
\]
Then $x_{\rho}+y_{\rho}=M_{\rho}-1$. For each $k\ge k_0$, divide $D_k$ into two subsets
\[
  D_k=D_{1,k}\cup D_{2,k},
\]
with 
\[
|D_{1,k}| = \begin{cases}
 x_{\rho}, & \text{if } k \in E, \\
 y_{\rho}, & \text{if } k\in O,
\end{cases}
\qquad\text{and}\qquad
|D_{2,k}| = \begin{cases}
 y_{\rho}, & \text{if } k \in E, \\
 x_{\rho}, & \text{if } k\in O,
\end{cases}
\]
and define
\[
  B_i=\bigcup_{k\ge k_0}D_{i,k}
  \quad(i=1,2),
  \qquad
  C=C_E\cup C_O.
\]
Then $A=B_1\cup B_2\cup C$ is a partition of $A$. To verify the required conditions in Theorem \ref{thm:strong-3set}, we need to prove
\begin{equation}\label{eq:double-Delta}
  \Delta(B_1),\ \Delta(B_2)<\infty
\end{equation}
and 
\begin{equation}\label{eq:C-L2-divergence}
H_1(C)=\{0\},\quad H_2(C)\text{ is countable}.
\end{equation}

For \eqref{eq:double-Delta}, note that 
if $m_{i,k}:=|B_i\cap (\rho^k,\rho^{k+1}]|$ for every $k\ge k_0$, then
\[
m_{i,k}+m_{i,k+1}\rho\ge y_{\rho}+x_{\rho}\,\rho.
\]
If $M_{\rho}=2u_{\rho}+1$, then $x_{\rho}=y_{\rho}=u_{\rho}\ge\rho(\rho-1)$, so that
\[m_{i,k}+m_{i,k+1}\ge \rho(\rho^2-1);\]
if $M_{\rho}=2v_{\rho}$, then $x_{\rho}=v_{\rho}-1$ and $y_{\rho}=v_{\rho}$, which implies
\[m_{i,k}+m_{i,k+1}\ge (1+\rho)v_{\rho}-\rho\ge\rho(\rho^2-1).\]
In either case, Lemma \ref{lem:Delta(S)} applied to each $B_i$ confirms \eqref{eq:double-Delta}.

To verify \eqref{eq:C-L2-divergence}, observe that 
\[H_1(C)=H_1(C_E)\cap H_1(C_O),\]
which, in conjunction with \eqref{eq:C_O-divergence}, implies that $H_1(C)=\{0\}$. Moreover, since $C\cap[\rho^k,\rho^{k+1}]\ne\emptyset$ for every $k\ge k_0$, so that when enumerated as a strictly increasing sequence, $C$ has bounded ratios less than $\rho^2$. By Lemma \ref{lem:bounded-ratio}, $H_2(C)$ is countable.

Invoking Theorem \ref{thm:strong-3set} and \eqref{eq:q_A} completes the proof of the base case $M=M_{\rho}$ of Theorem \ref{thm:main}.

Next, consider the general case $M\ge M_{\rho}$. For every $k\ge k_0$, choose an arbitrary $c_k\in A_k$, and put $C_0=\{c_k\}_{k\ge k_0}$. The same argument shows that $H_1(C_0)$ is countable, and an application of Lemma \ref{lem:countable-deletion} to $X_k=A_k':=A_k\setminus\{c_k\}$ for $k\ge k_0$, with $r_k=M-M_{\rho}<|A_k'|$ and $f_{\theta}(a)=\|a\theta\|$ for $\theta\in H_1(C_0)\setminus\{0\}$, yields sets $U_k\subseteq A_k'$ with $|U_k|=M-M_{\rho}$ such that
\begin{equation}\label{eq:T-divergence}
\sum_{k\ge k_0}\sum_{a\in A_k'\setminus U_k}\|a\theta\|=\infty,
  \quad\forall \theta\in H_1(C_0)\setminus\{0\}.        
\end{equation}
Let
\[
  U:=\bigcup_{k\ge k_0}U_k,
  \qquad
  V:=\bigcup_{k\ge k_0}(A_k\setminus U_k).
\]
Then $A=U\cup V$ is a partition of $A$. Note that 
\[
\big|U\cap (\rho^k,\rho^{k+1}]\big|=M-M_{\rho},
  \qquad
\big|V\cap (\rho^k,\rho^{k+1}]\big|\ge M_{\rho},
\]
for $k\ge k_0$. Since 
\[
V=C_0\cup\bigcup_{k\ge k_0}(A_k'\setminus U_k),
\]
\eqref{eq:T-divergence} implies
\[
H_1(V)=H_1(C_0)\cap H_1(A_k'\setminus U_k)=\{0\}.
\]
Thus, the base case $M=M_{\rho}$ of Theorem \ref{thm:main} that we just proved shows that $V$ is strongly complete. Applying Lemma \ref{lem:q_A} with $B=U$, $C=V$, and $h=M-M_{\rho}$ proves the general case $M\ge M_{\rho}$ of the theorem.
\end{proof}

\begin{rmk}\label{rmk:M_rho*}
Recall the definition of $M_{\rho}^{\ast}$ given by \eqref{eq:M_rho*}. It would be interesting to pin down its exact value. Theorem \ref{thm:main} shows that $M_{\rho}^{\ast}\le M_{\rho}$. On the other hand, it is possible to prove $M_{\rho}^{\ast}\ge u_{\rho}$ when $\rho\ge2$. To see this, suppose first that $\rho=2$ and consider the set $A=\{2^k+1:k\in\N\}$. Clearly, $|A\cap[2^{k},2^{k+1}]|=u_{\rho}=1$ for all $k\in\N$. In addition, if $\theta\in H_1(A)$, then $\|(2^k+1)\theta\|\to0$ and $\|(2^{k+1}+1)\theta\|\to0$ as $k\to\infty$. The triangle inequality gives
\[\|\theta\|=\|2(2^k+1)\theta-(2^{k+1}+1)\theta\|\le 2\|(2^k+1)\theta\|+\|(2^{k+1}+1)\theta\|\to0\]
as $k\to\infty$. So $\theta=0$. Hence, $H_1(A)=\{0\}$. However, $A$ is incomplete as $|A\cap[1,2^{k+1}]|=k$ implies $|\FS(A)\cap[1,2^{k+1}]|\le 2^{k}-1$, resulting in at least $2^k+1$ unrepresented numbers in $[1,2^{k+1}]$. 

In the case $\rho>2$, let $r=u_{\rho}-1\ge2$ and $I_k=(\rho^k,\rho^{k+1}]$ for $k\in\N$. Since $r<\rho(\rho-1)$, we may choose $\epsilon>0$ so small that
\begin{equation}\label{eq:rho-epsilon}
	\rho-3\epsilon>\frac{r(1+\epsilon)}{\rho-1}.    
\end{equation}
Suppose that $\epsilon\rho^k\ge r+1$ for every $k>k_0$ and fix $m\in\N$. For $1\le k\le k_0$, let $A_k=I_k\cap\N$; for $k>k_0$, let $A_k$ consist of $r$ consecutive elements in $(\rho^k,(1+\epsilon)\rho^k]$ when $k\not\equiv k_0\pmod*{m}$ and of $r$ consecutive elements in $((\rho-\epsilon)\rho^k,\rho^{k+1}]$ when $k\equiv k_0\pmod*{m}$. Put
\[
A=\bigcup_{k\ge 1}A_k.
 \]
 For every $k>k_0$ with $k\equiv k_0\pmod*{m}$, \eqref{eq:rho-epsilon} implies that the sum of all elements of $A$ not exceeding $(\rho-\epsilon)\rho^k$ is at most
  \[
   r\sum_{\ell=1}^{k-m}\rho^{\ell+1}+r(1+\epsilon)\sum_{\ell=k-m+1}^{k-1}\rho^{\ell}<\left(\frac{r\rho^{2-m}}{\rho-1}+\frac{r(1+\epsilon)}{\rho-1}\right)\rho^k<\left(\rho-2\epsilon\right)\rho^k,
   \]
 provided that $m$ is sufficiently large in terms of $\rho$ and $\epsilon$. Consequently, $\FS(A)$ misses every integer in $[(\rho-2\epsilon)\rho^k,(\rho-\epsilon)\rho^k]$. Hence, $A$ is incomplete. Furthermore, $A$ contains infinitely many pairs of consecutive integers, which leads to $H_1(A)=\{0\}$ as before.
 
 By a more involved argument, we are able to show that a random set $A$, formed by drawing $M$ elements uniformly from $(\rho^k,\rho^{k+1}]$ and independently for each $k$, is strongly complete almost surely when $M\ge u_{\rho}$ and incomplete almost surely when $M<u_{\rho}$, which seems to suggest that $u_{\rho}$ may be very close to the true value of $M_{\rho}^{\ast}$.
\end{rmk}

\begin{rmk}\label{rmk:M_2*}
The optimal case $M_2^{\ast}=2$ would imply that the set $A_{\alpha,\beta}$ defined by \eqref{eq:A_{alpha,beta}} is strongly complete under Hegyv\'{a}ri's condition that $\alpha\not\sim\beta$ and at least one of them is not a dyadic rational. Indeed, without loss of generality, we may assume that $\alpha$ is not a dyadic rational. For any $x>0$ and $K\ge0$, let
\[U_{K}(x):=\{\lfloor 2^kx\rfloor:k\ge K\}.\]
The condition that $\alpha\not\sim\beta$ ensures that $U_0(\alpha)\cap U_0(\beta)$ is finite. After rescaling by powers of 2 and isolating finitely many elements, we may assume
\[A_{\alpha,\beta}=U_{k_0}(\alpha)\cup U_{k_0}(\beta)\cup B\]
is a partition of $A_{\alpha,\beta}$, where $k_0\in\N$ can be arbitrarily large, $B\subseteq \N$ is finite depending on $k_0$, and $\alpha,\beta\in(1/2,1]$ are distinct with $\alpha$ not a dyadic rational. If $k_0$ is large enough, then $\lfloor 2^{k+1}\alpha\rfloor,\lfloor 2^{k+1}\beta\rfloor\in(2^{k},2^{k+1}]$ are distinct for $k\ge k_0$, so that
\[\big|A_{\alpha,\beta}\cap(2^{k},2^{k+1}]\big|\ge2,\quad\forall k\ge k_0.\] 
Moreover, the condition that $\alpha'$ is not a dyadic rational implies that $\lfloor 2^{k+1}\alpha'\rfloor=2\lfloor 2^{k}\alpha'\rfloor+1$ for infinitely many $k$, from which $H_1(A)=\{0\}$ follows for $A_{\alpha,\beta}$ by a routine application of the triangle inequality. Thus our assumption that $M_2^{\ast}=2$ implies that $A_{\alpha,\beta}$ is strongly complete.
\end{rmk}

The construction of the partition in the proof of Theorem \ref{thm:main} also yields the following strong-precompleteness result.

\begin{prop}\label{prop:strong-precomplete}
	Fix $\rho>1$ and suppose that $H_1(A)=\{0\}$ and
	\[
	|A\cap(\rho^k,\rho^{k+1}]|\ge u_\rho+1
	\]
	for every sufficiently large $k$. Then, for every finite $F\subseteq A$, the set $\FS(A\setminus F)$ is both thick and syndetic.
\end{prop}

\begin{proof}
	Fix a finite $F\subseteq A$ and put $A'=A\setminus F$. Let $k_0\in\N$ be an integer such that $k_0>\max F$ and $|A\cap(\rho^k,\rho^{k+1}]|\ge u_\rho+1$ for every $k\ge k_0$. Then $A'\cap(\rho^k,\rho^{k+1}]$ also contains at least $u_\rho+1$ elements for each $k\ge k_0$, which implies $\Delta(A')<\infty$ by Lemma \ref{lem:Delta(S)}. Hence, $\FS(A')$ is syndetic by Lemma \ref{lem:burr-erdos3.2}.
	
	Furthermore, repeating the proof of Theorem \ref{thm:main} above, but with $A'$ in place of $A$, with $r_k=u_\rho$ in the two applications of Lemma \ref{lem:countable-deletion}, and without splitting $D_k$ into two subsets, we obtain a partition $A'=B\cup C$ with $B,C$ infinite, such that $\Delta(B)<\infty$, $H_1(C)=\{0\}$, and $H_2(C)$ is countable. By Lemma \ref{lem:burr-erdos3.2}, $\FS(B)$ is syndetic, so $d^*(\FS(B))>0$. Lemmas \ref{Lem:thick-sumset} and \ref{lem:Hartman-H_1-H_2} together show that $\FS(B)+\FS(C)$ is thick. Therefore, $\FS(A')\supseteq\FS(B)+\FS(C)$ is also thick.
\end{proof}

The proof of Proposition \ref{prop:strong-precomplete} shows that a direct upgrade of Theorem \ref{thm:strong-3set} to a two-component strong-completeness criterion would immediately improve Theorem \ref{thm:main} with $u_\rho+1$ in place of $M_{\rho}$, and in particular, giving $M_2^*\le 3$, but such a direct upgrade seems difficult to attain. On the other hand, the next proposition affirms that there is no hope of deducing $M_2^*\le 4$ from Theorem \ref{thm:strong-3set}, suggesting that a genuinely different mechanism for strong completeness may be required.

\begin{prop}\label{prop:four-point-architecture}
	There is a strongly complete set $A\subseteq\N$ such that
	\[
	\big|A\cap(2^k,2^{k+1}]\big|=4
	\]
	for every $k\ge 3$, but there are no pairwise disjoint infinite subsets $B_1,B_2,C\subseteq A$ satisfying
	\[
	\Delta(B_1),\Delta(B_2)<\infty
	\qquad\text{and}\qquad
	H_1(C)=\{0\}.
	\]
\end{prop}

\begin{proof}
	Let $\Qq:=\{j^2:j\ge3\}$.  For $k\ge3$, define
	\[
	A_k:=
	\begin{cases}
		\{2^{k+1}-6,2^{k+1}-4,2^{k+1}-2,2^{k+1}\},
		&\text{if }k\in\Qq,\\
		\{2^k+1,2^k+2,2^k+4,2^k+6\},
		&\text{if }k+1\in\Qq,\\
		\{2^k+2,2^k+4,2^k+6,2^k+8\},
		&\text{otherwise},
	\end{cases}
	\qquad
	\]
	and put
	\[A:=\bigcup_{k\ge3}A_k.\]
	The three cases of $A_k$ are disjoint and all contained in $(2^k,2^{k+1}]$. In particular, 
	\[
	\big|A\cap(2^k,2^{k+1}]\big|=|A_k|=4,\qquad\forall k\ge3.
	\]
	Note also that the odd numbers in $A$ are exactly of the form $2^k+1$ with $k+1\in\Qq$ and that every $A_k$ contains at least three consecutive even numbers. For each integer $K\ge3$, let
	\[
	S_K:=\sum_{a\in A\cap[1,2^K]}a.
	\]
	It is easy to see that if $K\in\Qq$ is sufficiently large, then
	\begin{equation}\label{eq:S_K}
		S_K=4\sum_{\substack{k=3\\k,k+1\notin\Qq}}^{K-1}\left(2^k+O(1)\right)+4\sum_{\substack{k\in\Qq\\k<K}}\left(2^k+O(1)\right)=(4+o(1))2^K.    
	\end{equation}
	
	% Let $\theta\in H_1(A)\setminus\{0\}$. For each $k\ge3$, choose even numbers $x_k,y_k\in A_k$ with $y_k-x_k=2$. Then
	% \[
	%  \|2\theta\|
	%  \le\|x_k\theta\|+\|y_k\theta\|
	%  \to0
	% \]
	% as $k\to\infty$. Hence $2\theta=0$ in $\T$, or equivalently, $\theta=1/2$ in $\T$. The odd numbers in $A$ are exactly of the form $2^k+1$ with $k+1\in\Qq$. So
	% \[
	% \sum_{a\in A}\|a\theta\|\ge\sum_{\substack{k\ge3\\k+1\in\Qq}}\|(2^k+1)\theta\|=\infty,
	% \]
	% contradicting $\theta\in H_1(A)\setminus\{0\}$. It follows that $H_1(A)=\{0\}$.
	
	Suppose, to the contrary, that there were pairwise disjoint infinite $B_1,B_2,C\subseteq A$ with the asserted properties. In particular, the condition $H_1(C)=\{0\}$ forces $C$ to contain infinitely many of the odd numbers $2^k+1$ with $k+1\in\Qq$, since otherwise $1/2\in H_1(C)$.
	
	Let $B:=B_1\cup B_2$. Fix an arbitrarily large $K\in\Qq$ for which $2^{K-1}+1\in C$, and put, for each $i=1,2$,
	\[
	b_i:=\min\big(B_i\cap(2^K,\infty)\big),
	\]
	which exists since $B_i$ is infinite. Then $b_i\ge 2^{K+1}-6$. By the definition of $\Delta(B_i)$,
	\[
	\sum_{\substack{b\in B_i\\b\le2^K}}b
	=\sum_{\substack{b\in B_i\\b<b_i}}b
	\ge b_i-\Delta(B_i)
	\ge 2^{K+1}-6-\Delta(B_i).
	\]
	Therefore
	\begin{equation}\label{eq:sum-B-2^K}
		\sum_{b\in B\cap[1,2^K]}b
		\ge \left(2^{K+1}-6-\Delta(B_1)\right)+\left(2^{K+1}-6-\Delta(B_2)\right)=4\cdot2^K-O(1).
	\end{equation}
	On the other hand, since $B_1,B_2,C$ are pairwise disjoint and $2^{K-1}+1\in C$, \eqref{eq:S_K} gives
	\[
	\sum_{b\in B\cap[1,2^K]}b
	\le S_K-\left(2^{K-1}+1\right)
	=\left(\frac{7}{2}+o(1)\right)2^K,
	\]
	contradicting \eqref{eq:sum-B-2^K} for sufficiently
	large $K$. Hence, no such $B_1,B_2,C$ exist.
	
	Finally, strong completeness of $A$ is an immediately consequence of Corollary \ref{cor:3-bounded-gaps} below.
\end{proof}

Next, we deduce Theorem \ref{thm:bounded-ratio} from Theorem \ref{thm:strong-3set} using similar ideas. For any nonempty $S,X\subseteq\N$ arranged in increasing order, we say that $S$ has \textit{bounded index gaps} in $X$ if it meets every $\ell$ consecutive terms of $X$ for some fixed $\ell\in\N$, and denote by $\ell_X(S)$ the least such $\ell$ with the convention that $\ell_X(S)=\infty$ if $S$ has unbounded index gaps in $X$. Note that given $S\subseteq S'$, $X\subseteq X'$, and $T,Y$, we have $\ell_{X}(S')\le\ell_{X}(S)$, $\ell_{X}(S)\le\ell_{X'}(S)$, and $\ell_{X\cup Y}(S\cup T)\le\ell_X(S)+\ell_Y(T)-1$, the last of which follows by the pigeonhole principle and is clearly sharp.

\begin{proof}[Proof of Theorem \ref{thm:bounded-ratio}]
For any finite $F\subseteq A=\{a_1<a_2<\cdots\}$, the set $A\setminus[a_1,\max F]\subseteq A\setminus F$ satisfies the same conditions as $A$ after relabeling. Thus it suffices to prove that $A$ is complete. By assumption,
\begin{equation}\label{eq:fixed-step-bounded-ratios}
a_{n+1}<a_{n+r}\le \rho a_n+o(\Sigma_n)
\end{equation}
for all sufficiently large $n$, where $\Sigma_n:=\sum_{k\le n}a_k$. It follows that $\Sigma_n\le \Sigma_{n+1}\le (\rho+o(1))\Sigma_n$.

Let $S\subseteq A$ be infinite and has bounded index gaps in $A$. Then $S$ meets every $\ell$ consecutive terms of $A$ for some $\ell\in\N$. For all sufficiently large $n$,
\begin{equation}\label{eq:S-mass}
\sum_{\substack{s\in S\\s<a_n}}s=\sum_{k<n/\ell}\sum_{\substack{s\in S\\a_{n-k\ell}\le s<a_{n-(k-1)\ell}}}s\ge\frac{1}{\ell}\sum_{k<n/\ell}\,\sum_{0\le j<\min\{\ell,n-k\ell\}}a_{n-k\ell-j} =\frac{\Sigma_{n-\ell}}{\ell}\ge\frac{\Sigma_n}{\ell(\rho+o(1))^{\ell}}.
\end{equation}
We partition $S$ into consecutive finite sets $S_j$ with sums $v_j$ as follows. Start with $S_1=\{\min S\}$. Given $S_j$ with $\max S_j=a_{m}$, choose the minimal $n>m$ such that $a_n\in S$ and
\[\sum_{\substack{s\in S\\a_m<s\le a_n}}s\ge2v_j,\]
and put $S_{j+1}=S\cap(a_m,a_n]$. Then
\[2v_j\le v_{j+1}<2v_j+a_n.\]
Combining this with \eqref{eq:S-mass} yields
\[a_n\le \Sigma_n\ll_{\rho,\ell}\sum_{\substack{s\in S\\s<a_n}}s=\sum_{i=1}^j v_i+(v_{j+1}-a_n)<v_j\sum_{i=1}^j 2^{i-j}+2v_j<4v_j.\]
Hence, $2v_j\le v_{j+1}\ll_{\rho,\ell} v_j$. In particular, the sequence $V=\{v_j\}_{j\ge1}$ has bounded successive ratios. By Lemma \ref{lem:bounded-ratio}, $H_1(\{v_j:j\ge1\})$ is countable. The triangle inequality gives $H_1(S)\subseteq H_1(\{v_j:j\ge1\})$. Hence, $H_1(S)$ is also countable.

Put $d=3r=3(2\lceil\rho\rceil-1)$. We seek disjoint sets $A_0,U\subseteq A$ with bounded index gaps in $A$, such that $H_1(A_0)=\{0\}$ and the indices of consecutive terms of $A_0\cup U\subseteq A$ are at least $d$ apart. To this end, consider the partition
\[A=\bigcup_{k=1}^{\infty}I_k,\qquad I_k:=\{a_{1+3(k-1)d},\ldots,a_{3kd}\},\]
and let $D$ and $E$ consist of the $d$-th term and the $2d$-th term of each $I_k$, respectively:
\[D=\{a_{(3k-2)d}:k\in\N\},\qquad E=\{a_{(3k-1)d}:k\in\N\},\]
both of which have bounded index gaps in $A$. Thus $H_1(D)$ and $H_1(E)$ are countable. Put 
 \[f_{\theta}(k):=\sum_{a\in I_k}\|a\theta\|,
 \quad\forall\theta\in H_1(D)\setminus\{0\}.\] 
Then $H_1(A)=\{0\}$ implies that
\[\sum_{k\ge1} f_{\theta}(k)=\infty,\quad\forall \theta\in H_1(D)\setminus\{0\}.\] 
Applying Lemma \ref{lem:countable-deletion} to $X_j=\{2j-1,2j\}$ with $r_j=1$ for $j\in\N$ and $f_{\theta}(k)$ for $\theta\in H_1(D)\setminus\{0\}$, we find $k_j\in X_j$ for each $j\in\N$ such that
\[
\sum_{a\in Y}\|a\theta\|=\sum_{j\ge1}f_\theta(k_j)=\infty,
\quad\forall\theta\in H_1(D)\setminus\{0\},
\]
where
\[Y:=\bigcup_{j\ge1}I_{k_j}.\]
Hence $H_1(Y)\cap H_1(D)=\{0\}$. Let $T=D\setminus Y$ and $U=E\setminus Y$. Since $T$ has bounded index gaps in $A$ with $\ell_A(T)\le 9d$, $H_1(T)$ is countable. In addition, we have
\begin{equation}\label{eq:H_1(T)-H_1(Y)}
H_1(T)\cap H_1(Y)=H_1(T\cup Y)=H_1(D\cup Y)= H_1(Y)\cap H_1(D)=\{0\}.
\end{equation}

For $Y$, we may first partition $K=\{k_j: j\in\N\}$ as
\[K=\bigcup_{i=1}^{\infty}K_i,\qquad \min K_{i+1}\ge \max K_i+2,\]
where each $K_i$ is a maximal string of consecutive integers of length one or two contained in $K$, and then partition $Y$ accordingly as
\[Y=\bigcup_{i=1}^{\infty}Y_i,\qquad Y_i:=\bigcup_{k\in K_i}I_k.\]
Each $Y_i$ satisfies $3d\le |Y_i|\le 6d$. Moreover, \eqref{eq:H_1(T)-H_1(Y)} gives
\[\sum_{a\in Y}\|a\theta\|=\infty,
\quad\forall\theta\in H_1(T)\setminus\{0\}.\]
Thus Lemma \ref{lem:countable-deletion} applied to $X_i=Y_i$ with $r_i=|Y_i|-1$ and $f_\theta(a)=\|a\theta\|$ for $\theta\in H_1(T)\setminus\{0\}$ produces some $y_i\in Y_i$ for each $i\in\N$ such that
\[
\sum_{i\ge1}\|y_i\theta\|=\infty,
\quad\forall\theta\in H_1(T)\setminus\{0\}.
\]
Let 
\[A_0:=T\cup \{y_i:i\ge1\}.\]
Then $A_0$ has bounded index gaps in $A$ with $\ell_A(A_0)\le\ell_A(T)\le 9d$ and
\[H_1(A_0)=H_1(T)\cap H_1(\{y_i:i\ge1\})=\{0\}.\] 
In addition, 
\[A_0\cup U=\{y_i:i\ge1\}\cup\bigcup_{\substack{k=1\\k\notin K}}^{\infty}\{a_{(3k-2)d},a_{(3k-1)d}\}.\]
Since $K_i$ and $K_{i+1}$ are separated by some $I_k$ with $\max K_i<k<\min K_{i+1}$, it follows immediately that the indices of consecutive terms of $A_0\cup U\subseteq A$ are at least $d$ apart, as desired.

Let $n_1<n_2<\cdots$ be the indices for which $a_{n_k}\notin A_0\cup U$. Define
\[
B_1=\{a_{n_{2k-1}}:k\ge1\},
\qquad
B_2=\{a_{n_{2k}}:k\ge1\}.
\]
Put $t=\lceil\rho\rceil-1$, so that $r=2t+1$. For $a_n\in B_i$ and $1\le j\le 3$, at most one of the preceding $jr$ terms $a_{n-jr},\ldots,a_{n-1}$ belongs to $A_0\cup U$, since the indices of consecutive terms of $A_0\cup U\subseteq A$ are at least $d=3r$ apart. Thus at least 
\[\left\lfloor\frac{r-1}{2}\right\rfloor=\begin{cases}
t, &\text{if }j=1,\\
2t,&\text{if }j=2,\\
3t+1,&\text{if } j=3,
\end{cases}\]
of them belong to $B_i$. Hence
\[
\sum_{\substack{b\in B_i\\a_{n-3r}\le b<a_n}} b
=\sum_{j=1}^3\sum_{\substack{b\in B_i\\a_{n-jr}\le b<a_{n-(j-1)r}}}b\ge ta_{n-r}+ta_{n-2r}+ (t+1)a_{n-3r}
\ge\ a_n+o_{\rho}(\Sigma_n)
\]
for sufficiently large $n$, where the last inequality follows from \eqref{eq:fixed-step-bounded-ratios}. Since $B_i$ meets every three consecutive terms of $A$, \eqref{eq:S-mass} implies that
\[\sum_{\substack{b\in B_i\\ b<a_{n-3r}}} b\gg_{\rho}\Sigma_n.\]
Therefore
\[a_n-\sum_{\substack{b\in B_i\\ b<a_n}} b<0\]
for sufficiently large $n$. So $\Delta(B_i)<\infty$.

Now we partition $U$ into consecutive finite sets $U_j$ with sums $v_j$ as described earlier, and put $V=\{v_j:j\ge1\}$. Since $V$ has bounded successive ratios, so does $A\cup V$. Let $C=A_0\cup(V\setminus A)$. Then $H_1(C)\subseteq H_1(A_0)=\{0\}$ and $B_1,B_2,C$ are pairwise disjoint. Moreover, 
\[\ell_{A\cup V}(C)\le \ell_{A}(A_0)+\ell_{V\setminus A}(V\setminus A)-1\le 9d.\]
If $A\cup V$ has bounded successive ratios at most $c_0>1$, then  $C$ has bounded successive ratios at most $c_0^{9d}$. Thus $H_2(C)$ is countable by Lemma \ref{lem:bounded-ratio}. By Theorem \ref{thm:strong-3set}, $A'=B_1\cup B_2\cup C$ is complete. Since $\FS(A')\subseteq\FS(A)$, so is $A$.
\end{proof}

We end this section with a proof of Theorem \ref{thm:complete-decomposition} based on Theorem \ref{thm:main}.
\begin{proof}[Proof of Theorem \ref{thm:complete-decomposition}]
Since (2) trivially implies (3), the equivalence between (1)--(3) will follow from (1)$\implies$(2) and (3)$\implies$(1).

We first prove the easier direction (3)$\implies$(1). Assume (3). Since $A_1$ is strongly complete and $A_1\subseteq A$, we know that $A$ is also strongly complete. In particular, $H_1(A)=\{0\}$, or equivalently, \eqref{eq:L1-divergence} holds. Moreover, \eqref{eq:A_j-growth} implies that 
\begin{equation}\label{eq:A_j-m_k}
\big|A_j\cap(\rho^k,\rho^{k+1}]\big|=\left(\pi_j+o_j(1)\right)m_k,\qquad\forall j\ge1.    
\end{equation}
Since $m_k\ge1$ for sufficiently large $k$, we have $A_j\cap(\rho^k,\rho^{k+1}]\ne\emptyset$ for sufficiently large $k$. Fixing $J\in\N$ and summing \eqref{eq:A_j-m_k} over $1\le j\le J$ yield
\[
\left(1+o_J(1)\right)m_k\ge\left(\sum_{j=1}^{J}\pi_j+o_J(1)\right)m_k=\sum_{j=1}^{J}\big|A_j\cap(\rho^k,\rho^{k+1}]\big|\ge J,
\]
which implies $m_k\ge(1+o_J(1))J$ for sufficiently large $k$ in terms of $J$. Since $J$ is arbitrary, $m_k\to\infty$ as $k\to\infty$, proving (1).

Now we prove (1)$\implies$(2) by arguing probabilistically. Fix $\{\pi_j\}_{\ge1}$ and choose $k_0\in\N$ so that $m_k\ge2$ for every $k\ge k_0$. Put
\[
 A^*:=\bigcup_{k\ge k_0}X_k.
\]
Since $H_1(A)=\{0\}$ and $A\setminus A^*$ is finite, we have $H_1(A^*)=\{0\}$. 

Suppose that $\{n_{j,k}\}_{j\ge1}$ is a sequence of nonnegative integers summing up to $m_k$ such that
\begin{equation}\label{eq:n_{j,k}/m_k}
 \lim_{k\to\infty}\frac{n_{j,k}}{m_k}=\pi_j,\qquad\forall j\ge1.
\end{equation}
Of course, only finitely many $n_{j,k}$ are nonzero for any fixed $k$. For each $k\ge k_0$, independently choose a partition 
\[
 X_k=\bigcup_{j\ge1}A_{j,k},\quad\text{where}\quad|A_{j,k}|=n_{j,k},
\]
uniformly at random among all the 
\[
\binom{m_k}{n_{1,k},n_{2,k},\ldots}:=\frac{m_k!}{\prod_{j\ge1}n_{j,k}!} 
\]
possible partitions of labeled elements of $X_k$ into sets of sizes $n_{1,k},n_{2,k},\ldots$. Conditionally on the partition, independently choose a singleton $C_{j,k}$
of each nonempty $A_{j,k}$ and put $B_{j,k}:=A_{j,k}\setminus C_{j,k}$. If $A_{j,k}=\emptyset$, set $B_{j,k}=C_{j,k}=\emptyset$. Let
\begin{equation}\label{eq:B_j-C_j}
B_j:=\bigcup_{k\ge k_0}B_{j,k} \qquad\text{and}\qquad C_j:=\bigcup_{k\ge k_0}C_{j,k}.
\end{equation}
Fix $j\ge1$. Since $\pi_j>0$ and $m_k\to\infty$ as $k\to\infty$, \eqref{eq:n_{j,k}/m_k} implies that there is an integer $K_j\ge k_0$ such that 
\begin{equation}\label{eq:n_{j,k}-m_k}
n_{j,k}-1\ge\frac{\pi_j}{2}(m_{k}-1)\ge 1    
\end{equation}
whenever $k\ge K_j$. In particular, $C_j$, when enumerated as a strictly increasing sequence, has bounded ratios, so Lemma \ref{lem:bounded-ratio} makes $H_1(C_j)$ countable. 

Put $c_0:=2(1-\e^{-1/2})$, so that $1-\e^{-t}\ge c_0t$ for $t\in[0,1/2]$. If $S$ is an $r$-element subset of a finite set $Y$, where $|Y|=R$, chosen uniformly at random and $f:Y\to[0,1/2]$, then Maclaurin's inequality on mean values of elementary symmetric polynomials gives
\begin{align}\label{eq:app-Maclaurin}
 \E\left[\exp\left(-\sum_{y\in S}f(y)\right)\right]
 =\frac{1}{\binom{R}{r}}\sum_{\substack{S\subseteq Y\\|S|=r}}
       \prod_{y\in S}\e^{-f(y)}\nonumber
 &\le\left(\frac{1}{R}\sum_{y\in Y}\e^{-f(y)}\right)^r\\
 &\le\exp\left(-c_0\frac{r}{R}\sum_{y\in Y}f(y)\right).
\end{align}
Let $\theta\in H_1(C_j)\setminus\{0\}$. Since
$H_1(A^*)=\{0\}$, we have
\begin{equation}\label{eq:X_k-C_{j,k}-L1-divergence}
\sum_{k\ge k_0}\,\sum_{a\in X_k\setminus C_{j,k}}\|a\theta\|=\infty.    
\end{equation}
Conditional on $C_{j,k}$, the set $B_{j,k}$ can be thought of as an $(n_{j,k}-1)$-element subset of $X_k\setminus C_{j,k}$ chosen uniformly at random. By \eqref{eq:n_{j,k}-m_k},
\[
\frac{|B_{j,k}|}{|X_k\setminus C_{j,k}|}
 =\frac{n_{j,k}-1}{m_k-1}\ge\frac{\pi_j}{2}
\]
for $k\ge K_j$. Thus for any $K\ge K_j$, independence and \eqref{eq:app-Maclaurin} applied with $Y=X_k\setminus C_{j,k}$, $S=B_{j,k}$, and $f(y)=\|y\theta\|$, imply 
\begin{align*}
\E\left[
 \left.\exp\left(-\sum_{k=K_j}^K\sum_{a\in B_{j,k}}\|a\theta\|\right)\right|C_j\right]
&=\prod_{k=K_j}^{K}\E\left[
 \left.\exp\left(-\sum_{a\in B_{j,k}}\|a\theta\|\right)\right|C_{j,k}\right]\\
 &\le\prod_{k=K_j}^{K}\exp\left(-c_0\frac{n_{j,k}-1}{m_k-1}\sum_{a\in X_k\setminus C_{j,k}}\|a\theta\|\right)\\
 &\le\exp\left(-c_0\frac{\pi_j}{2}\sum_{k=K_j}^{K}\sum_{a\in X_k\setminus C_{j,k}}\|a\theta\|\right), 
\end{align*}
Combining this with Markov's inequality, we have, for every fixed $L>0$,
\begin{align*}
\PP\left(\left.\sum_{k=K_j}^{\infty}\sum_{a\in B_{j,k}}\|a\theta\|<L\,\right|C_j\right)&\le \PP\left(\left.\sum_{k=K_j}^K\sum_{a\in B_j}\|a\theta\|<L\,\right|C_j\right)\\
&\le e^L\,\E\left[
 \left.\exp\left(-\sum_{k=K_j}^L\sum_{a\in B_{j,k}}\|a\theta\|\right)\right|C_j\right]\\
&\le\exp\left(L-c_0\frac{\pi_j}{2}\sum_{k=K_j}^{K}\sum_{a\in X_k\setminus C_{j,k}}\|a\theta\|\right).
\end{align*}
In view of \eqref{eq:X_k-C_{j,k}-L1-divergence}, we obtain, by first letting $K\to\infty$ and then $L\to\infty$,
\[
\PP\left(\left.\sum_{k=K_j}^{\infty}\sum_{a\in B_{j,k}}\|a\theta\|<\infty\,\right|C_j\right)=0.
\]
Since this holds for every $\theta\in H_1(C_j)\setminus\{0\}$, the union bound yields 
\begin{equation}\label{eq:conditional-C_j-divergence}
\PP\left(\left.\sum_{k=K_j}^{\infty}\sum_{a\in B_{j,k}}\|a\theta\|=\infty,~\forall\theta\in H_1(C_j)\setminus\{0\}\right|C_j\right)=1.    
\end{equation}
Let $E_j$ be the event that $H_1(B_j\cup C_j)=\{0\}$. Then
\[
E_j=\bigcap_{q=2}^{\infty}\bigcap_{L=1}^{\infty}\bigcup_{k=k_0}^{\infty}\left\{\min_{\substack{\theta\in\T\\\|\theta\|\ge1/q}}\,\sum_{k=k_0}^{K}\sum_{a\in A_{j,k}}\|a\theta\|>L\right\},
\]
where we observe that $A_{j,k}=(B_j\cup C_j)\cap X_k$. In particular, $E_j$ is measurable. Since $H_1(B_j\cup C_j)=H_1(B_j)\cap H_1(C_j)$, the event $E_j|\,C_j$ contains the event in \eqref{eq:conditional-C_j-divergence}. Hence
$\PP(E_j|\,C_j)=1$. Averaging over all the choices for $C_j$, we get $\PP(E_j)=1$. In other words, 
\[
H_1(B_j\cup C_j)=\{0\}     
\]
occurs almost surely. Also, 
\[
\lim_{k\to\infty}|(B_j\cup C_j)\cap X_k|=\lim_{k\to\infty}|A_{j,k}|=\lim_{k\to\infty}n_{j,k}=\infty
\]
for each $j\ge1$. By Theorem \ref{thm:main},
$B_j\cup C_j$ is strongly complete almost surely for each $j\ge1$.

Now we construct the desired partition of $A$. For $j\ge1$ and $k\ge k_0$, let
\[
 \Pi_0:=0,\qquad \Pi_j:=\sum_{i=1}^{j}\pi_i,
 \qquad
 n_{j,k}:=\lceil m_k\Pi_j\rceil-\lceil m_k\Pi_{j-1}\rceil.
\]
Then
\[
\sum_{j\ge1}n_{j,k}=\lim_{J\to\infty}\sum_{j=1}^{J}n_{j,k}=\lim_{J\to\infty}\lceil m_k\Pi_J\rceil=m_k.    
\]
In addition,
\begin{equation}\label{eq:approx-n_{j,k}/m_k}
 \big|n_{j,k}-\pi_jm_k\big|<1,
\end{equation}
so \eqref{eq:n_{j,k}/m_k} is fulfilled. Thus the discussion above shows that for each $j\ge1$, there is a realization of the sets $B_j,C_j$ in \eqref{eq:B_j-C_j} such that for $j\ge1$ and $k\ge k_0$, $A_j^*:=B_j\cup C_j$ is strongly complete,
\[
|A_j^*\cap X_k|=|(B_j\cup C_j)\cap X_k|=n_{j,k}
\]
satisfies \eqref{eq:approx-n_{j,k}/m_k}, and 
\[
A^*=\bigcup_{j\ge1}A_j^*
\]
is a partition. For each $0\le k<k_0$, choose an arbitrary partition 
\[
X_k=\bigcup_{j\ge1}A_{j,k},\quad\text{where}\quad|A_{j,k}|=n_{j,k},
\]
and put
\begin{align*}
A_1&:=(A\cap\{1\})\cup A_1^*\cup\bigcup_{k=1}^{k_0-1}A_{1,k},\\
A_j&:=A_j^*\cup\bigcup_{k=1}^{k_0-1}A_{j,k},\qquad\forall j\ge2.
\end{align*}
Then 
\[
A=\bigcup_{j\ge1}A_j
\]
is a partition into strongly complete sets $A_j$ satisfying \eqref{eq:A_j-growth}. This verifies (1)$\implies$(2).
\end{proof}

\section{An ordered-block generalization of Theorem \ref{thm:main}}\label{S:block-ext-main}
The $\rho$-adic intervals in the condition \eqref{eq:growth} produce the bounded-ratio sequences $S_O,S_E$, which lead to countable spectra and retained divergence, while ensuring the reservoirs $B_1,B_2$ receive sufficiently many elements so that they have bounded $\Delta$-values. Once having these two features isolated, we can extend Theorem \ref{thm:main}, replacing these intervals with more general windows. 

\begin{thm}\label{thm:ordered-blocks}
Let $A\subseteq\N$ satisfy $H_1(A)=\{0\}$, and let
$\{A_k\}_{k\ge 1}$ be a sequence of finite nonempty subsets of $A$ such that 
\[
  \max A_k<\min A_{k+1},\quad\forall k\ge 1,
\]
and $A\setminus\bigcup_{k\ge 1}A_k$ is finite. For each $k\ge 1$, define
\[
  \alpha_k:=\min A_k,\qquad\beta_k:=\max A_k,
\]
and assume that there exist $k_0,r\in\N$ and sequences of nonnegative integers $\{r_{1,k}\}_{k\ge k_0}$ and $\{r_{2,k}\}_{k\ge k_0}$, each positive for infinitely many $k$, such that 
\[
r_{1,k}+r_{1,k}=r,\quad\forall k\ge k_0,
\]
\[
 |A_k|\ge r+1,\quad\forall k\ge k_0,
\]
\begin{equation}\label{eq:ordered-ratio}
\sup_{k\ge k_0}\frac{\alpha_{k+1}}{\alpha_k}<\infty,
\end{equation}
and
\[
 \sup_{k\ge k_0}
 \left(\beta_k-\sum_{\ell=k_0}^{k-1}r_{i,\ell}\alpha_\ell\right)<\infty,\quad\forall i=1,2.
\]
Then $A$ is strongly complete.
\end{thm}

% To see why Theorem \ref{thm:ordered-blocks} subsumes the qualitative part of Theorem \ref{thm:main}, take $A_k=A\cap(\rho^k,\rho^{k+1}]$ for $k\ge 1$, and set $r=u_\rho$. Then \eqref{eq:A_k-growth} is exactly \eqref{eq:growth} with $M=M_{\rho}$, while $\alpha_{k+1}/\alpha_k<\rho^2$. Furthermore, since $u_{\rho}\ge\rho(\rho-1)$, we find that
% \[\beta_k-r\sum_{\ell=k_0}^{k-1}\alpha_\ell\le\rho^{k+1}-\rho(\rho-1)\sum_{\ell=k_0}^{k-1}\rho^\ell=\rho^{k+1}-\rho\big(\rho^{k}-\rho^{k_0}\big)=\rho^{k_0+1}.\]
% Thus, all the conditions in Theorem \ref{thm:ordered-blocks} are fulfilled, so Theorem \ref{thm:main} follows.

\begin{proof}
It is enough to work with
$A^*=\bigcup_{k\ge k_0}A_k$, since $A\setminus A^*$ is finite and $H_1(A^*)=\{0\}$. Arguing as in the proof of Theorem \ref{thm:main} with the observation that \eqref{eq:ordered-ratio} makes bounded-ratio sequences possible everywhere they are needed, we obtain a partition $A^*=B_1\cup B_2\cup C$ such that $H_1(C)=\{0\}$, $H_2(C)$ is countable, and each block $B_i\cap A_k~(k\ge k_0)$ contains $r_{i,k}$ elements. Setting
\[
  \Delta_0=\max_{i=1,2}\sup_{k\ge k_0}\left(\beta_k-\sum_{\ell=k_0}^{k-1}r_{i,\ell}\alpha_\ell\right)<\infty,
\]
we have
\[
  b-\sum_{\substack{b'\in B_i\\b'<b}}b'\le b-\sum_{\ell=k_0}^{k-1}\,\sum_{b'\in B_i\cap A_{\ell}}b'\le\beta_k-\sum_{\ell=k_0}^{k-1}r_{i,\ell}\alpha_\ell\le\Delta_0
\]
for every $b\in B_i\cap A_k$, where $i=1,2$ and $k\ge k_0$. So $\Delta(B_1),\ \Delta(B_2)<\infty$. Theorem \ref{thm:strong-3set} then implies that $A^*$ is strongly complete, and hence so is $A$.
\end{proof}

\begin{cor}\label{cor:consecutive-blocks}
Let $k_0,r\in\N$, and let $\{b_k\}_{k\ge 1}$ be a strictly increasing sequence of positive integers satisfying $b_k+2r<b_{k+1}$ for all $k\ge k_0$, 
\[\sup_{k\ge k_0}\frac{b_{k+1}}{b_k}<\infty,\]
and
\[
\sup_{k\ge k_0}\left(b_k-r\sum_{\ell=k_0}^{k-1}b_\ell\right)<\infty.
\]
Then
\[
  A=\bigcup_{k\ge k_0}\{b_k,b_k+1,\ldots,b_k+2r\}
\]
is strongly complete.
\end{cor}

\begin{proof}
This follows from Theorem \ref{thm:ordered-blocks} applied to 
\[
  A_k=\{b_k,b_k+1,\ldots,b_k+2r\},
\]
once we verify $H_1(A)=\{0\}$, or equivalently, \eqref{eq:L1-divergence}. But if the series in \eqref{eq:L1-divergence} converged at some $\theta\in\T\setminus\{0\}$, then $\|b_k\theta\|\to0$ and $\|(b_k+1)\theta\|\to0$ as $k\to\infty$, so that
\[
  \|\theta\|\le\|b_k\theta\|+\|(b_k+1)\theta\|\to0
\]
as $k\to\infty$, which would imply $\theta=0$, a contradiction.
\end{proof}

Taking $r=1$ and $b_k=2^k+1$ in Corollary \ref{cor:consecutive-blocks} gives the strongly complete set
\[
  A=\bigcup_{k\ge2}\big\{2^k+1,2^k+2,2^k+3\big\}
\]
is strongly complete. However, this is not covered by Theorem \ref{thm:main} for any fixed $\rho>1$. To see this, note that
\[
  \big|A\cap[1,\rho^{k}]\big|=(3\log_2\rho)\,k+O(1),\quad\forall k\ge1.
\]
If \eqref{eq:growth} held for $A$ and some $\rho>1$, then it would give
\[\big|A\cap[1,\rho^{k}]\big|\ge M_{\rho}\,k+O(1)\]
for sufficiently large $k$. Hence $M_\rho\le 3\log_2\rho$,
which is impossible, since $M_\rho\ge2>3\log_2\rho$ for $1<\rho<2^{2/3}$, $M_\rho\ge3>3\log_2\rho$ for $2^{2/3}\le\rho<2$, and $M_\rho\ge2\rho(\rho-1)+1>3\log_2\rho$ for $\rho\ge2$. 

For a second family, fix $t>0$, $r\in\N$, and
$\alpha\in(1,r+1)$. With $b_k=\lfloor t\alpha^k\rfloor$, we have 
\[
 b_k-r\sum_{\ell<k}b_\ell
 =t\alpha^k\left(1-\frac{r}{\alpha-1}\right)+O(k),
\]
which is bounded above because $\alpha<r+1$. Hence, Corollary \ref{cor:consecutive-blocks} implies that
\[
 \bigcup_{k\ge 1}
 \{\lfloor t\alpha^k\rfloor,
   \lfloor t\alpha^k\rfloor+1,\ldots,
   \lfloor t\alpha^k\rfloor+2r\}
\]
is strongly complete. This is much weaker than Theorem \ref{thm:poly-talpha^n} that we will establish in the next section.

\section{Polynomially perturbed ray sets: Proof of Theorem \ref{thm:poly-talpha^n}}\label{S:poly-talpha^n}

We dedicate this section to a proof of Theorem \ref{thm:poly-talpha^n} on strong completeness of the polynomially perturbed single-ray set
\[
 A_{t,\alpha,\Pp}:=\N\cap\{\lfloor t\alpha^n\rfloor+P_i(n):n\in\N,1\le i\le k\},
\] 
where $t,\alpha>0$ and $\Pp=\{P_1,\ldots,P_k\}$ is a set of integer-valued polynomials. The proof does not follow directly from Theorem \ref{thm:main} or Theorem \ref{thm:ordered-blocks}; instead, it combines Graham's theorem on completeness of polynomial sequences with a paired switching argument. We mention that Burr \cite{Burr79} studied a different problem in which the polynomially perturbed sequences also have polynomial growth, whereas our sequence $\{\lfloor t\alpha^n\rfloor\}_{n\in\N}$ grows exponentially.

For every $N\in\N$ and every sequence $S=\{s_j\}_{j\ge1}$ of positive integers, define
\[
 \FS_{N}(S):=\left\{\sum_{j\in J}s_j:
     \emptyset\ne J\subseteq \{j\ge N\}\text{ finite}\right\}.
\]

We begin with the following lemma, partly inspired by \cite[Lemma 5]{Graham64a}, in which a set of disjoint pairs provides a key ingredient for the construction of thick subset sums needed for strong completeness.

\begin{lem}\label{lem:paired-complete}
Let $B\subseteq\N$ be infinite with $\Delta(B)<\infty$. For $n\in\N$, let $x_n<y_n$ be positive integers such that the pairs $\{x_n,y_n\}$ are pairwise disjoint and disjoint from $B$, and put $s_n=y_n-x_n$ and $S=\{s_n\}_{n\ge1}$. Suppose that there is some $d\in\N$ such that the following hold:
\begin{enumerate}
 \item for every $N\in\N$, the set $\FS_{N}(S)$ contains arbitrarily long arithmetic progressions with common difference $d$;
 \item for every $N\in\N$,
 \[
  \gcd\big(\{b\in B:b\ge N\}\cup\{d\}\big)=1.
 \]
\end{enumerate}
Then
\[
 A=B\cup\{x_n,y_n:n\ge1\}
\]
is strongly complete.
\end{lem}

\begin{proof}
Since deleting finitely elements from $B$ and finitely many pairs $\{x_n,y_n\}$ affects neither the finiteness of $\Delta(B)$ by \eqref{eq:Delta(S-T)} nor the conditions (1) and (2), it suffices to show that $A$ is complete.

Applying Proposition \ref{prop:modular} in Appendix with $\Pp=\{p:p\mid d\}$, we see that (2) implies
\[\FS(B)\pmod*{d}=\Z/d\Z.\]
In particular, there is a finite set $E\subseteq B$ such that
\begin{equation}\label{eq:FS(E)-mod d}
\FS(E)\pmod*{d}=\Z/d\Z.
\end{equation}

Let $C=\{x_n,y_n:n\ge1\}$. We claim that
$\FS(E)+\FS(C)$, where the empty sum from $E$ is allowed, is thick. The proof is essentially the same as that of \cite[Lemma 5]{Graham64a}. Note that condition (1) supplies, for infinitely many $M\in\N$, arithmetic progressions
\[
\{a,a+d,\ldots,a+Md\}\subseteq \FS_1(S),  
\]
where $a=a_M\in\N$. Take a finite set $I\subseteq\N$ of indices such that
\begin{equation}\label{eq:progression-FS_1(S)}
\{a,a+d,\ldots,a+Md\}\subseteq\FS(\{s_n:n\in I\}),
\end{equation}
and let $X_I:=\sum_{n\in I}x_n$. If 
\[
 x\in[X_I+a+\max \FS(E),\,X_I+a+Md+\min \FS(E)],
\]
then
\[
 x-X_I\in[a+\max \FS(E),\,a+Md+\min \FS(E)].
\]
By \eqref{eq:FS(E)-mod d}, we can find $H\in\FS(E)$ and $k\in\Z$ such that
\[
 x-X_I-H=a+kd\in[a+\max \FS(E)-H,\,a+Md+\min \FS(E)-H]\subseteq[a,a+Md].
\]
According to \eqref{eq:progression-FS_1(S)}, there is a finite set $J\subseteq I$ such that 
\[x-X_I-H=a+kd=\sum_{n\in J}s_n,\]
which implies that
\[x=X_I+H+\sum_{n\in J}s_n=H+\sum_{n\in I\setminus J}x_n+\sum_{n\in J}y_n\in\FS(E)+\FS(C),\]
since the pairs $\{x_n,y_n\}$ are pairwise disjoint and disjoint from $E\subseteq B$. Hence
\[
 [X_I+a+\max \FS(E),\,X_I+a+Md+\min \FS(E)]\subseteq\FS(E)+\FS(C).
\]
This verifies our claim. 

Since $\Delta(B)<\infty$ implies $\Delta(B\setminus E)<\infty$ by \eqref{eq:Delta(S-T)}, Lemma \ref{lem:burr-erdos3.2} shows that $\FS(B\setminus E)$ is syndetic. It follows that
\[\FS(A)\supseteq\FS(B\setminus E)+\FS(E)+\FS(C)\]
is cofinite, meaning that $A$ is complete.
\end{proof}

Albeit not needed for the proof of Theorem \ref{thm:poly-talpha^n}, we record two interesting applications of Lemma \ref{lem:paired-complete}. The first concerns finite translates with primitive difference sets.

\begin{cor}\label{cor:4-point-translates}
Fix $\rho>1$, $k_0\in\N$, and a finite set $F\subseteq\Z$ such that $\gcd(F-F)=1$ and 
\[
|F|\ge \min\{u_\rho+2,2v_{\rho}\},
\]
where $u_{\rho},v_{\rho}$ are defined as in \eqref{eq:constants}. Let $\{b_k\}_{k\ge k_0}\subseteq\Z$ satisfy
$b_k+F\subseteq(\rho^k,\rho^{k+1}]$ for every $k\ge k_0$, and put
\[
 A:=\bigcup_{k\ge k_0}(b_k+F).
\]
Then $A$ is strongly complete. 
\end{cor}

\begin{proof}
Suppose first that $|F|\ge 2v_{\rho}$. Then $|A\cap(\rho^k,\rho^{k+1}]|=|F|\ge M_{\rho}$. By Theorem \ref{thm:main}, it suffices to verify $H_1(A)=\{0\}$. Let $\theta\in H_1(A)$. Then for every $t\in F$, we have $\|(b_k+t)\theta\|\to0$ as $k\to\infty$. Consequently, the triangle inequality gives
\[
  \|(t-t')\theta\|
  \le
  \|(b_k+t)\theta\|+\|(b_k+t')\theta\|
  \to0
\]
as $k\to\infty$ for every pair $t,t'\in F$. Hence, $\|(t-t')\theta\|=0$. Since this holds for every pair $t,t'\in F$ and $\gcd(F-F)=1$, we have $\theta=0$. Therefore $H_1(A)=\{0\}$.

Now consider the case $|F|\ge u_{\rho}+2$. Enumerate all the doubletons of $F$ as $F_1,\ldots,F_m$, where $m=\binom{|F|}{2}\ge3$. For each $k\ge k_0$, take the unique index $1\le i_k\le m$ with $i_k\equiv k\pmod*{m}$. Put
\begin{align*}
C&:=\bigcup_{k\ge k_0}(b_k+F_{i_k}),\\
B&:=\bigcup_{k\ge k_0}(b_k+F\setminus F_{i_k}),
\end{align*}
and $S:=\{s_k\}_{k\ge k_0}$ with $s_k=y_k-x_k$, where $x_k<y_k$ constitute $F_{i_k}$. 

Note that the assumption $\gcd(F-F)=1$ implies that
\[\gcd\left(\{s_k:k\ge K\}\right)=1,\quad\forall K\in\N,\]
since every positive difference in $F-F$ appears in $S$ infinitely many times. By the solution to the Frobenius coin problem, $\FS_K(S)$ contains all sufficiently large integers for every $K\in\N$.

Similarly, we have
\[\gcd\big(B\cap[N,\infty)\big)=1\]
for every $N\in\N$. Besides, since 
\[
|B\cap(\rho^k,\rho^{k+1}]|\ge u_{\rho},\quad\forall k\ge k_0,
\]
Lemma \ref{lem:Delta(S)} gives $\Delta(B)<\infty$. 

Finally, Lemma \ref{lem:paired-complete} applied to the partition $A=B\cup C$ with $d=1$ proves the corollary.
\end{proof}

\begin{rmk}
The quantity $2u_{\rho}+1$, which defines $M_{\rho}$ together with $2v_{\rho}$, is left out of the assumption on $|F|$ because it agrees with $u_{\rho}+2$ when $1<\rho\le(1+\sqrt{5})/2$ but is strictly larger when $\rho>(1+\sqrt{5})/2$.
\end{rmk}

The second application of Lemma \ref{lem:paired-complete} gives an improvement of Theorem \ref{thm:main} when the elements of $A$ have bounded gaps infinitely often.

\begin{cor}\label{cor:3-bounded-gaps}
Let $\rho>1$. A set $A\subseteq\N$ is strongly complete if it satisfies
\begin{enumerate}[label=(\roman*)]
    \item $|A\cap(\rho^k,\rho^{k+1}]|\ge u_{\rho}+1+\mathbf{1}_{\rho>2}$ for all sufficiently large $k$;
    \item $\gcd(A\cap[N,\infty))=1$ for every $N\in\N$;
    \item assuming $A=\{a_1<a_2<\cdots\}$, 
    \[
    \liminf_{n\to\infty}(a_{n+1}-a_n)<\infty.
    \]
\end{enumerate}
\end{cor}

\begin{proof}
Let $I_k=(\rho^k,\rho^{k+1}]$ for $k\ge1$. By (iii), some fixed $d\in\N$ occurs infinitely often as the gap between consecutive elements of $A$. Let $k_0\in\N$ be such that $\rho^{k_0}>d$ and (i) holds for every $k\ge k_0$. There is some residue class $r\pmod*{d}$ for which there are infinitely many pairs $a_n,a_{n+1}=a_n+d$ with $a_n>\rho^{k_0}$ and $a_n\equiv r\pmod{d}$. Each of such pairs $a_n,a_{n+1}$ either lies within some $I_k$ or crosses a boundary in the sense that $a_n\in I_k$ and $a_{n+1}\in I_{k+1}$ for some $k\ge k_0$. Select an infinite family $\Ff$ of these pairs, making sure that each $(\rho^k,\rho^{k+1}]$ intersects at most one pair. Further split $\Ff$ into two infinite subfamilies $C$ and $D$ and put $B:=A\setminus C$. We will be slightly more specific about this split below.

It is evident that the sequence $S$ of the pair differences in $C$ is just an infinite string of the same number $d$. Thus $\FS_N(S)=d\N$ for every $N\in\N$.

It is also easy to see that for every $N\in\N$,
 \[
  \gcd\big(\{b\in B:b\ge N\}\cup\{d\}\big)=1.
 \]
Indeed, suppose that a prime $p\mid d$ divides every element of $B\cap[N,\infty)$. Since $D\subseteq B$, we have $p\mid r$. So $p$ also divides every element in $C$. It follows that $p$
divides every element of $A\cap[N,\infty)$, contradicting (ii).

It remains to show $\Delta(B)<\infty$. Suppose that $b\in B\cap(\rho^k,\rho^{k+1}]$ for some $k\ge k_0$. Write $\N_{\ge k_0}=E_1\cup E_2$, where
\begin{align*}
E_1&:=\big\{\ell\ge k_0: \big|(\rho^{\ell},\rho^{\ell+1}]\cap C\big|\le1\big\},\\
E_2&:=\big\{\ell\ge k_0: \big|(\rho^{\ell},\rho^{\ell+1}]\cap C\big|=2\big\}.
\end{align*}
By (i), 
\[
|B\cap(\rho^{\ell},\rho^{\ell+1}]|\ge u_{\rho}+1+\mathbf{1}_{\rho>2}-i\ge \rho(\rho-1)+1+\mathbf{1}_{\rho>2}-i
\]
for each $\ell\in E_i$. 

If $\Ff$ contains only finitely many pairs lying within some $I_k$, the split can be arbitrary. Since $E_2$ is finite, we have
\[
 b-\sum_{\substack{b'\in B\\b'<b}}b'
 \le\rho^{k+1}-\rho(\rho-1)\sum_{\ell=k_0}^{k-1}\rho^\ell+O_{k_0,\rho}(1)\ll_{k_0,\rho}1.
\]
If $\Ff$ contains infinitely many pairs lying within some $I_k$, it is possible to choose a split with the additional property that $C$ only contains pairs lying within some $I_k$ and is so sparse that when $E_2=\{\ell_{j}\}_{j\ge1}$ is expressed as a strictly increasing sequence, 
\begin{equation}\label{eq:l_j}
 \left(\frac{(1+\mathbf{1}_{\rho>2})\rho}{\rho-1}-2\right)\rho^{\ell_j}+\mathbf{1}_{\rho\in\N}(u_{\rho}+1+\mathbf{1}_{\rho>2})(\ell_j+1)\ge 2\sum_{i=1}^{j-1}\rho^{\ell_i}+2j\mathbf{1}_{\rho\in\N},
 \qquad\forall j\ge1,
\end{equation}
which is possible because at least one of the coefficients of $\rho^{\ell_j}$ and $\ell_j+1$ is positive. It follows that
\begin{align}\label{eq:b-sum-Delta(B)}
b-\sum_{\substack{t\in B\\t<b}}t&\le \rho^{k+1}
   -(u_{\rho}+1+\mathbf{1}_{\rho>2})\sum_{\substack{\ell=k_0\\\ell\ne\ell_j}}^{k-1}(\rho^{\ell_j}+\mathbf{1}_{\rho\in\N})
   -(u_{\rho}-1+\mathbf{1}_{\rho>2})\sum_{k_0\le \ell_j<k}(\rho^{\ell_j}+\mathbf{1}_{\rho\in\N})\nonumber\\
&\le \rho^{k+1}
   -(u_{\rho}+1+\mathbf{1}_{\rho>2})\sum_{\ell=k_0}^{k-1}(\rho^{\ell_j}+\mathbf{1}_{\rho\in\N})
   +2\sum_{k_0\le \ell_j<k}(\rho^{\ell_j}+\mathbf{1}_{\rho\in\N})\nonumber\\
&\le c_0-\left(\frac{u_{\rho}+1+\mathbf{1}_{\rho>2}}{\rho-1}-\rho\right)\rho^k-\mathbf{1}_{\rho\in\N}(u_{\rho}+1+\mathbf{1}_{\rho>2})k+2\sum_{k_0\le \ell_j<k}(\rho^{\ell_j}+\mathbf{1}_{\rho\in\N})\nonumber\\
&\le c_0-\frac{1+\mathbf{1}_{\rho>2}}{\rho-1}\rho^k-\mathbf{1}_{\rho\in\N}(u_{\rho}+1+\mathbf{1}_{\rho>2})k+2\sum_{k_0\le \ell_j<k}(\rho^{\ell_j}+\mathbf{1}_{\rho\in\N}),
\end{align}
where 
\[
c_0:=\frac{u_{\rho}+1+\mathbf{1}_{\rho>2}}{\rho-1}\rho^{k_0}+\mathbf{1}_{\rho\in\N}(u_{\rho}+1+\mathbf{1}_{\rho>2})k_0.
\]
The expression in \eqref{eq:b-sum-Delta(B)} decreases with $k$ except possibly at some $k=\ell_{j_0}+1\ge k_0$, in which case, it is at most
\[
 c_0-\left(\frac{(1+\mathbf{1}_{\rho>2})\rho}{\rho-1}-2\right)\rho^{\ell_{j_0}}-\mathbf{1}_{\rho\in\N}(u_{\rho}+1+\mathbf{1}_{\rho>2})(\ell_{j_0}+1)+2\sum_{j=1}^{j_0-1}\rho^{\ell_j}+2j_0\mathbf{1}_{\rho\in\N}\le c_0
\]
by \eqref{eq:l_j}. Hence $\Delta(B)<\infty$.

Thus all hypotheses of Lemma \ref{lem:paired-complete} hold, proving
strong completeness of $A$.
\end{proof}

Now we deduce Theorem \ref{thm:poly-talpha^n} from Lemma \ref{lem:paired-complete}. Recall that $\operatorname{Int}(\Z)$ denotes the ring of integer-valued polynomials and
\[
 \alpha_k:=\frac{k-1+\sqrt{(k-1)^2+8}}{2}
\]
for each $k\ge 2$. Observe that every $P\in\operatorname{Int}(\Z)$ is periodic modulo every $q\in\N$, since if $c\in\N$ makes $cP\in\Z[x]$, then $c(P(n+cq)-P(n))$ is divisible by $cq$ for every integer $n$.

\begin{proof}[Proof of Theorem \ref{thm:poly-talpha^n}]
Fix $t>0$, $1<\alpha<\alpha_k$, and $\Pp=\{P_1,\ldots,P_k\}\subseteq\operatorname{Int}(\Z)$, where $k\ge2$. Put $a_n:=\lfloor t\alpha^n\rfloor$ for $n\in\N$ and recall
\[
 d_{\Pp}:=\gcd\left(\{P_i(n)-P_j(n):n\in\Z,1\le i<j\le k\}\right).
\]

The implication (1)$\implies$(2) is immediate, since every strongly complete set has trivial $H_1$-spectrum as discussed in the introduction.

Next, we prove the equivalence of (2) and (3). Let
$\theta\in H_1(A_{t,\alpha,\Pp})$. As $n\to\infty$, 
\[
 \|(a_n+P_i(n))\theta\|\to0,\quad \forall 1\le i\le k,
\]
and hence for every distinct pair $i,j$, $\|(P_i(n)-P_j(n))\theta\|\to0$ as $n\to\infty$. This is impossible unless every $P_i$ is constant or $\theta\in\Q$. In either case, $\theta=a/b$ for some $a\in\Z$ and $b\in\N$ with $\gcd(a,b)=1$. Then for every distinct pair $i,j$, we have $b\mid a_n+P_i(n)$ and $b\mid P_i(n)-P_j(n)$ for all sufficiently large $n$. Since every $P_i$ is periodic modulo $b$, the second divisibility holds for every integer $n$, and hence $b\mid d_{\Pp}$. Now, if (3) holds, then $b=1$, and hence $\theta=0$ in $\T$, proving $H_1(A_{t,\alpha,\Pp})=\{0\}$. Conversely, if there is some prime $p\mid d_{\Pp}$ such that $p\mid a_n+P_1(n)$ for all sufficiently large $n$, then the same divisibility condition holds for $a_n+P_i(n)$ for every $1\le i\le k$, which implies that $1/p\in H_1(A_{t,\alpha,\Pp})\setminus\{0\}$. This verifies (2)$\iff$(3).  

It remains to prove (2)$\implies$(1). Let $Q:=P_1-P_2$ and
\[g:=\gcd(\{Q(n):n\in\Z\}).\]
Let $n_0\in\N$ and denote by $p_1,\ldots,p_{\ell}$ the distinct prime divisors of $g$. Condition (2) implies that for each $1\le m\le \ell$, there is some index $1\le i_m\le \ell$ such that
\[
 E_m:=\{n\ge n_0:p_m\nmid a_n+P_{i_m}(n)\}
\]
is infinite, since otherwise we would have $1/p_m\in H_1(A_{t,\alpha,\Pp})\setminus\{0\}$ for some $m$. Choose an integer $M>\max\{\ell,1\}$ so large that
\begin{equation}\label{eq:choice-m}
 \frac{k}{\alpha-1}-\frac{2}{\alpha(1-\alpha^{-M})}>1;
\end{equation}
such a choice is possible because when $1<\alpha<\alpha_k$, the expression on the left-hand side tends to 
\[ \frac{k}{\alpha-1}-\frac{2}{\alpha}=\frac{-\alpha^2+(k-1)\alpha+2}{\alpha(\alpha-1)}+1>1\]
as $M\to\infty$. For each $1\le m\le\ell$, let $c_m\pmod*{M}$ be a residue class to which infinitely many members of $E_m$ belong. Choose a residue class $r\pmod*{M}$ different from all the $c_m$, and put
\[
 R:=\{n\ge n_0:n\equiv r\pmod*{M}\}.
\]
Then $E_m\setminus R$ is infinite for every $1\le m\le \ell$.  

To show that $A_{t,\alpha,\Pp}$ is strongly complete, it suffices to prove that
\[A_{t,\alpha,\Pp}^*:=\bigcup_{i=1}^k\{a_n+P_i(n):n\ge n_0\}\]
is strongly complete, where $n_0$ is chosen so large that this is a partition of $A_{t,\alpha,\Pp}^*\subseteq\N$ such that the elements in each of its components are strictly increasing with $n$. The existence of such an $n_0$ is assured by the estimate $P(n)=O(n^D)=o(\alpha^n)$ for every $P\in\Pp$, where $D=\max_{P\in\Pp}\deg P$. In preparation for our application of Lemma \ref{lem:paired-complete}, define 
\[C:=\bigcup_{n\in R}\{x_n,y_n\}\]
and $B:=A_{t,\alpha,\Pp}^*\setminus C$, where 
\begin{align*}
x_n&:=\min\{a_n+P_1(n),a_n+P_2(n)\},\\
y_n&:=\max\{a_n+P_1(n),a_n+P_2(n)\},
\end{align*}
satisfy $y_n-x_n=|Q(n)|$. Then $A_{t,\alpha,\Pp}^*=B\cup C$ is a partition.

Let
\[
 h:=\gcd\{Q(n):n\equiv r\pmod*{M}\}=\gcd\{Q(r+Mt):t\in\mathbb Z\}.
\]
Choose a large $t_0\in\N$ and some $\epsilon\in\{-1,1\}$ such that $\epsilon Q(r+M(t+t_0))>0$ for all $t\ge0$. The polynomial
\[
Q^*(t):=\frac{\epsilon Q(r+M(t+t_0))}{h}\in\operatorname{Int}(\Z)
\]
is positive for $t\ge0$ and has positive leading coefficient and trivial fixed divisor. It follows that for every $T\ge 0$, $\FS_T(\{Q^*(t)\}_{t\ge0})$ contains all sufficiently large integers; see \cite{RothSzekeres54} and \cite[Theorem 1 and Concluding remarks]{Graham64a}. As a consequence, given any $N\in\N$, $\FS_N(\{y_n-x_n\}_{n\in R})=\FS_N(\{|Q(n)|\}_{n\in R})$ contains all sufficiently large multiples of $h$. So condition (1) of Lemma \ref{lem:paired-complete} is fulfilled with $d=h$.

We next verify condition (2) of Lemma \ref{lem:paired-complete}. Let $p\mid h$ be prime. If $p\mid g$, then $p=p_m$ for some $1\le m\le\ell$. Since $E_m\setminus R$ is infinite, $p\nmid a_n+P_{i_m}$ for infinitely many $n\notin R$. If $p\nmid g$, periodicity of $Q$ modulo $p$ gives infinitely many $n$ for which $p\nmid Q(n)$. Since $p\mid h$, these $n$ all lie outside $R$. Moreover, for each of these $n$ at least one of $a_n+P_1(n),a_n+P_2(n)$ is indivisible
by $p$. This proves
\[
 \gcd\big(\{b\in B:b\ge N\}\cup\{g\}\big)=1,\quad\forall N\in\N.
\]

Finally, we show that $\Delta(B)<\infty$. Let $b\in \{a_n+P_1(n),\ldots,a_n+P_k(n)\}$ for some $n\ge n_0$. Since
\[
k\sum_{s=n_0}^{n-1}\alpha^s-2\sum_{\substack{s=n_0\\n\in R}}^{n-1}\alpha^s
\ge k\sum_{s=n_0}^{n-1}\alpha^s-2\sum_{s\ge0}\alpha^{n-1-sM}\ge\alpha^n\left(\frac{k}{\alpha-1}
       -\frac{2}{\alpha(1-\alpha^{-M})}\right)+O(1),
\]
we have
\begin{align*}
b-\sum_{\substack{b'\in B\\b'<b}}b'&\le b-\left(\sum_{i=1}^k\sum_{s=n_0}^{n-1}(a_s+P_i(s))-\sum_{\substack{s=n_0\\s\in R}}^{n-1}(x_s+y_s)\right)\\
&\le t\alpha^n
       \left[1-\left(\frac{k}{\alpha-1}
       -\frac{2}{\alpha(1-\alpha^{-m})}\right)\right]
       +O\big(n^{D+1}\big),
\end{align*}
which is bounded above for $n\ge n_0$ by \eqref{eq:choice-m}. Therefore $\Delta(B)<\infty$.

Hence, Lemma \ref{lem:paired-complete} applied to $A_{t,\alpha,\Pp}^*=B\cup C$ proves (1).
\end{proof}

\begin{rmk}\label{rmk:H_q(A_{t,alpha,P})}
The proof of (2)$\iff$(3) above works perfectly well with $H_q(A_{t,\alpha,\Pp})$ in place of $H_1(A_{t,\alpha,\Pp})$ for any $q>0$. In fact, it shows that if we define
\[
 d_N:=\gcd\big(\{d_{\Pp}\}\cup\{\lfloor t\alpha^n\rfloor+P_1(n):n\ge N\}\big)
\]
for each $N\in\N$, which is nondecreasing under divisibility and bounded by $d_{\Pp}$, and is therefore eventually equal to some constant $d_\infty\in\N$, then
\[
 H_q(A_{t,\alpha,\Pp})=\frac{1}{d_\infty}\Z/\Z\subseteq\T.
\]
\end{rmk}

\begin{rmk}
Theorem \ref{thm:poly-talpha^n} is false if $\alpha>k+1$. This follows from the simple observation that $\alpha>k+1$ leads to
\[
 \sum_{i-1}^k\sum_{s=1}^{n-1}(\lfloor t\alpha^s\rfloor+P_i(s))=\left(\frac{k}{\alpha-1}+o(1)\right)t\alpha^n
 <\min_{1\le i\le k}\{\lfloor t\alpha^n\rfloor+P_i(n)\}
\]
for all sufficiently large $n$. It would be interesting to determine the exact values of $t>0$ and $\alpha\in[\alpha_k, k+1]$ for which $A_{t,\alpha,\Pp}$ is strongly complete.
\end{rmk}

\section{Mixed power sets: Proof of Theorem \ref{thm:BS-3-sets}}\label{S:power-sets}
We close our discussion with a short proof of Theorem \ref{thm:BS-3-sets} based on Theorem \ref{thm:strong-3set}. Fix $k\in\N$ and a nonempty finite set $D\subseteq\N_{\ge2}$ with no two elements being powers of the same integer. Suppose that $D=D_1\cup D_2\cup C$ is a partition of $D$ into nonempty subsets such that $\gcd(C)=1$ and \eqref{eq:sum-D_i} holds for $i=1,2$. We show that 
\[
 \Pp_k(D)=\bigcup_{d\in D}\{d^n:n\ge k\}
\]
is strongly complete. 

Let $\Bb_i:=\Pp_k(D_i)$ for $i=1,2$ and put $\Cc:=\Pp_k(C)$.
Our hypothesis on $D$ guarantees that $\Pp_k(D)=\Bb_1\cup\Bb_2\cup\Cc$ is a partition. Moreover, if $b\in \Bb_i$, then we have
\[
 \sum_{\substack{n\ge k\\ d^n<b}}d^n
 \ge \frac{b-d^k}{d-1},\qquad\forall d\in D_i.
\]
Indeed, this is trivial when the sum on the left-hand side is empty. When the sum is nonempty, it is equal to
\[
\sum_{n=k}^{m}d^n=\frac{d^{m+1}-d^k}{d-1}\ge\frac{b-d^k}{d-1},
\]
where $m$ is the largest integer for which $d^m<b$. It follows by \eqref{eq:sum-D_i} that
\[
 b-\sum_{\substack{b'\in \Bb_i\\b'<b}}b'\le b\left(1-\sum_{d\in D_i}\frac1{d-1}\right)
   +\sum_{d\in D_i}\frac{d^k}{d-1}
 \le \sum_{d\in D_i}\frac{d^k}{d-1}.
\]
We have hence proved $\Delta(\Bb_i)<\infty$ for $i=1,2$.

Let $c_0=\min C$. Then every $c\in \Cc$ satisfies $c\ge c_0^k$, and the least element in $\Cc$ that is strictly greater than $c$ is at most $c_0c$. Consequently, when enumerated as a strictly increasing sequence, $\Cc$ has bounded successive ratios, and Lemma \ref{lem:bounded-ratio} then implies that $H_2(\Cc)$ is countable.

We next show that $H_1(\Cc)=\{0\}$. Suppose that
$\theta\in H_1(\Cc)$. Then $\|c^n\theta\|\to0$ as $n\to\infty$ for every $c\in C$. Let $r_n(c)\in\Z$ be such that $\|c^n\theta\|=|c^n\theta-r_n(c)|$. Then
\[
 |r_{n+1}(c)-cr_n(c)|
 \le \| c^{n+1}\theta\|
       +c\| c^n\theta\|<1
\]
for all sufficiently large $n$, and hence $r_{n+1}(c)=cr_n(c)$ for all sufficiently large $n$. This implies that $\|c^{n+1}\theta\|=c\|c^n\theta\|\to0$ eventually. So there exists some $N_c\in\N$ such that $c^n\theta=0$ in $\T$ for all $n\ge N_c$. Necessarily, $\theta\in\Q$. Suppose that $\theta=a/q$ for some $q\in\N$ and $a\in\Z$ with $\gcd(a,q)=1$. Then $q\mid c^{N_c}$ for every $c\in C$.
Hence, every prime divisor of $q$ divides every member of $C$. Since $\gcd(C)=1$, we have $q=1$ and therefore $\theta=0$ in $\T$, as desired.

Finally, Theorem \ref{thm:strong-3set} applied to the partition $\Pp_k(D)=\Bb_1\cup\Bb_2\cup\Cc$ implies that $\Pp_k(D)$ is strongly complete.

\section{Appendix}

In this appendix we show that \eqref{eq:C-modular} follows from \eqref{eq:rational-theta}, or equivalently, from \eqref{eq:gcd-tail}. This implication is contained in the following stronger result.

\begin{prop}\label{prop:modular}
Let $C\subseteq\N$ be an infinite set and define 
\[C_N(q):=(C\cap[N,\infty))\cup\{q\}\] 
for any $N,q\in\N$. Let $\Pp$ be a set of primes and put
\[\N_{\Pp}:=\{q\in\N:p\mid q\implies p\in\Pp\}.\]
Then $\gcd(C_N(q))=1$ for every $N\in\N$ and every $q\in\N_{\Pp}$ if and only if 
\[
  \FS(C\setminus D)+q\Z=\Z
\]
for every finite $D\subseteq C$ and every $q\in\N_{\Pp}$. 
\end{prop}

\begin{proof}
The sufficiency part follows from the observation that if $\gcd(C_N(q))=d$, then $d\in\N_{\Pp}$ and $\FS(C\setminus D)\subseteq d\Z$ with $D=C\cap[1,N)$.

For the necessity part, it suffices to consider $C$ itself, for if $D\subseteq C$ is finite, then every tail of $C\setminus D$ has a trivial gcd with $q$ as well. The case $q=1$ is trivial. Fix $q\in\N_{\Pp}\cap[2,\infty)$ and write $G=\Z/q\Z$. Then for every $N\in\N$, $C\cap[N,\infty)$ modulo $q$ generates $G$, since otherwise the generated subgroup would have the form $d\Z/q\Z$ for some divisor $d>1$ of
$q$, and so every element of $C\cap[N,\infty)$ would be divisible by $d$, contradicting $\gcd(C_N(q))=1$.

We may enumerate elements of $C$ increasingly as $c_1<c_2<\cdots$. For each $n\in\N$, let $S_n\subseteq G$ be the set of subset sums of $c_1,\ldots,c_n$ modulo $q$, including the empty set sum $0$. Since $G$ is finite, the increasing sequence
\[
  S_1\subseteq S_2\subseteq\cdots\subseteq G
\]
eventually stabilizes. Thus there are $n_0\in\N$ and $S\subseteq G$ such that $S_n=S$ for every $n\ge n_0$. If $n>n_0$, then
\[
  S_n=S_{n-1}\cup(S_{n-1}+c_n),
\]
so $S+c_n\subseteq S$. Since $|S+c_n|=|S|$, we have $S+c_n=S$ for every $n>n_0$. Therefore, the stabilizer
\[
  G_S=\{g\in G:S+g=S\}
\]
is a subgroup of $G$ containing $C\cap[c_{n_0+1},\infty)$ modulo $q$. Since $C\cap[c_{n_0+1},\infty)$ modulo $q$ generates $G$, we conclude that $G_S=G$. As $0\in S$, this forces $S=G$. In particular, every nonzero $g\in G$ has a representation as a nonempty subset sum of $C$ modulo $q$.

Finally, we show that $0\in G$ can also be represented as a
nonempty subset sum of $C$ modulo $q$. To see this, choose any $q$ distinct elements of $C$ and consider the $q+1$ partial sums 
\[
  0,\ c_1,\ c_1+c_2,\ldots,\ c_1+\cdots+c_q
  \pmod*  q.
\]
By the pigeonhole principle, two of them are congruent, and their difference, which is a nonempty sum of distinct elements, is congruent to $0$ modulo $q$. 

This completes the proof.
\end{proof}

\medskip
\section*{AI disclosure}
ChatGPT 5.6 was used for proofreading the manuscript. It also suggested a core idea underlying the current shorter proof of Lemma \ref{lem:countable-deletion} which replaced the author's original probabilistic argument. The author takes full responsibility for the contents of the paper.

\end{document}